\documentclass[journal,twoside,web]{ieeecolor}

\usepackage{generic}
\usepackage{cite}
\usepackage{amsmath,amssymb,amsfonts}
\usepackage{algorithmic}
\usepackage{graphicx}
\usepackage{graphics}
\usepackage{tcolorbox}
\usepackage{url}
\usepackage{color}
\usepackage{hyperref}
\usepackage{textcomp}
\usepackage{mathtools, cuted}
\usepackage{lipsum}
\usepackage{stfloats}
\usepackage{multirow}
\usepackage{wrapfig}
\usepackage{caption}
\usepackage{subcaption}
\usepackage{algorithm}
\usepackage{algorithmic}
\usepackage[algo2e,linesnumbered,ruled]{algorithm2e}
\usepackage[utf8]{inputenc}
\usepackage{amsmath,amssymb,euscript,psfrag,latexsym,graphicx}
\usepackage{bbm,color,amstext,wasysym,cuted,mathtools, cite}
\usepackage[normalem]{ulem}
\usepackage{dsfont}
\usepackage{ulem}
\usepackage{mathtools,halloweenmath}
\usepackage{graphicx}  
\usepackage{epstopdf}

\definecolor{subsectioncolor}{rgb}{0.1, 0.1, 0.5}

\graphicspath{{./},{./figures/}}

\makeatletter
\let\proof\@undefined
\let\endproof\@undefined
\makeatother
\usepackage{amsthm}

\newcommand{\cL}{{\cal L}}
\newcommand{\cM}{{\cal M}}
\newcommand{\cN}{{\cal N}}
\newcommand{\cZ}{{\cal Z}}
\newcommand{\cC}{{\cal C}}
\newcommand{\cK}{{\cal K}}
\newcommand{\cF}{{\cal F}}

\newcommand{\cA}{{\cal A}}

\newcommand{\cP}{{\cal P}}

\newcommand{\cH}{{\cal H}}
\newcommand{\mC}{{\mathbb C}}

\newcommand{\mR}{{\mathbb R}}
\newcommand{\mE}{{\mathbb E}}

\newcommand{\mU}{{\mathbb U}}
\newcommand{\bG}{{\mathbf G}}
\newcommand{\bD}{{\mathbf D}}
\newcommand{\bQ}{{\mathbf Q}}
\newcommand{\bb}{{\mathbf b}}
\newcommand{\bV}{{\mathbf V}}
\newcommand{\bU}{{\mathbf U}}

\newcommand{\bs}{{\mathbf s}}

\newcommand{\bF}{{\mathbf F}}

\newcommand{\bP}{{\mathbf P}}
\newcommand{\bK}{{\mathbf K}}
\newcommand{\bp}{{\mathbf p}}
\newcommand{\bW}{{\mathbf W}}
\newcommand{\bM}{{\mathbf M}}

\newcommand{\bff}{{\mathbf f}}
\newcommand{\bh}{{\mathbf h}}
\newcommand{\bw}{{\mathbf w}}
\newcommand{\bx}{{\mathbf x}}

\newcommand{\bC}{{\boldsymbol{\mathcal C}}}

\newcommand{\bX}{{\mathbf X}}

\newcommand{\bE}{{\mathbf E}}

\newcommand{\bg}{{\mathbf g}}
\newcommand{\bA}{{\mathbf A}}
\newcommand{\bB}{{\mathbf B}}
\newcommand{\bu}{{\mathbf u}}
\newcommand{\bL}{{\mathbf L}}
\newcommand{\bd}{{\mathbf d}}
\newcommand{\bz}{{\mathbf z}}

\newcommand{\bR}{{\mathbf R}}
\newcommand{\bS}{{\mathbf S}}

\newcommand{\bJ}{{\mathbf J}}

\newcommand{\bc}{{\mathbf c}}

\newtheorem{approximation}{Approximation}
\newtheorem{theorem}{Theorem}
\newtheorem{definition}{Definition}
\newtheorem{lemma}{Lemma}
\newtheorem{assumption}{Assumption}
\newtheorem{remark}[theorem]{Remark}

\newtheorem{corollary}{Corollary}

\newtheorem{proposition}{Proposition}

\begin{document}

\title{When Koopman Meets Hamilton and Jacobi}

\author{ Umesh Vaidya, {\it IEEE Senior Member}
\thanks{
Financial support from NSF CPS award 1932458 and NSF 2031573 is greatly acknowledged. The author is with the Department of Mechanical Engineering, Clemson University, Clemson SC, 29631. }}

\maketitle


\begin{abstract}
In this paper, we establish a connection between the spectral theory of the Koopman operator and the solution of the Hamilton Jacobi (HJ) equation.  
The HJ equation occupies a central place in systems theory, and its solution is of interest in various control problems, including optimal control, robust control, and input-output analysis. A Hamiltonian dynamical system can be associated with the HJ equation and the solution of the HJ equation can be extracted from the Hamiltonian system in the form of Lagrangian submanifold. 
One of the main contributions of this paper is to show that the Lagrangian submanifolds can be obtained using the spectral analysis of the Koopman operator. We present two different procedures for the approximation of the HJ solution. We utilize the spectral properties of the Koopman operator associated with the uncontrolled dynamical system and Hamiltonian systems to approximate the HJ solution. We present a convex optimization-based computational framework with convergence analysis for approximating the Koopman eigenfunctions and the Lagrangian submanifolds.  
Our solution approach to the HJ equation using Koopman theory provides for a natural extension of results from linear systems to nonlinear systems.
We demonstrate the application of this work for solving the optimal control problem. 
Finally, we present simulation results to validate the paper's main findings and compare them against linear quadratic regulator and Taylor series based approximation controllers. 
\end{abstract}


\section{Introduction}
The  Hamilton Jacobi (HJ) equation is at the heart of several problems of interest in systems and control theory. The HJ equation arises in optimal control, robust $H_\infty$ control, dissipativity-based analysis of an input-output system, and control of systems with an adversary or min-max dynamic games\cite{van_der_shaft_book}. The HJ equation and its discrete-time counterpart, the HJ Bellman (HJB) equation, have attracted renewed attention due to the significance of this equation in data-driven control and Reinforcement learning problems \cite{bertsekas1996stochastic,sutton2018reinforcement}.
The HJ equation is a nonlinear partial differential equation (PDE), and given the significance of the HJ equation in systems theory, a variety of methods are developed for the approximation of its solution \cite{achdou2013hamilton}.    \\

\noindent{\it Literature review}: Given the nonlinear nature of the HJ PDE, the analytical solution is impossible, and one has to resort to a numerical scheme for its approximation. {\color{black} One of the popular numerical methods provides an iterative approach to solving the HJ equation. The iterative approach alternates between solving a linear PDE for the value function for a fixed control input and then updating the control input using the value function. The linear PDE is solved using Galerkin projection onto the finite-dimensional basis function for the approximate value function. \cite{beard1997galerkin}.} 
This iterative approach for solving optimal control problems via the HJB equation and the Bellman equation plays a fundamental role in various RL algorithms, including policy iteration, value iteration, and actor-critic method \cite{kumar2009computational}. 
Another line of research involving viscosity-based approximate solution to the HJ equation is proposed in \cite{crandall1992user,bardi2008optimal}. The viscosity-based solution is weaker than the classical, differentiable solution of the HJ equation. An approximate suboptimal solution of the HJ equation based on the series expansion of higher-order nonlinear terms is proposed in \cite{garrard1969additional,garrard1977design,lukes1969optimal,sannomiya1971method,navasca2000solution,al1961optimal}.\\ 
\noindent {\it Differential geometric viewpoint of HJ equation and Koopman theory}: An alternate approach for analyzing and approximating the HJ equation is based on a differential geometric-based interpretation of its solution. It is well known that a Hamiltonian dynamical system is associated with the HJ equation. The Lagrangian submanifolds of the Hamiltonian dynamical system, which are invariant manifolds, are used to construct the solution of the HJ equation. 
This differential geometric viewpoint is exploited to solve the optimal control, $H_\infty$ control, and ${\cal L}_2$ gain analysis and synthesis problems in \cite{van_der_shaft_book}. {\color{black} In \cite{sakamoto2008analytical}, the authors have exploited this differential geometric approach to develop computational methods for approximating the HJ solution. The methods we discovered in this paper for approximating HJ solution draw parallels to the techniques found in \cite{sakamoto2008analytical}. In particular, our first approach is similar to the approximation procedure developed in \cite{sakamoto2008analytical} relying on decomposing Hamiltonian into integrable and nonintegrable parts.}
In this paper, we show a strong connection between the differential geometric viewpoint of the HJ equation and the spectral analysis of the Koopman operator.  
The development of the Koopman operator was originally motivated for studying the ensemble or statistical behavior of conservative dynamical systems \cite{koopman1931hamiltonian}. However, the spectral properties of the Koopman operator have an intimate connection to the state space geometry \cite{mezic2017koopman,mezic2021koopman}. In particular, the invariant manifolds of the dynamical system are obtained as joint zero-level sets of the eigenfunctions. 
The recent work involving the Koopman operator for dynamical systems with dissipation provides a way of characterizing stable, unstable manifolds of nonlinear dynamical systems in terms of the zero-level sets of Koopman eigenfunctions \cite{mezic2020spectrum}. The explosion of research activities in Koopman theory provides for systematic data-driven and model-based methods for the computation of Koopman eigenfunctions \cite{housparse,klus2020eigendecompositions,korda2018convergence,korda2020optimal}. 
On the other hand, there is also extensive literature on the use of Koopman theory for control \cite{mauroy2016global,peitz2019koopman,villanueva2021towards,borggaard2009control, abraham2019active,korda2018linear, sootla2018optimal,otto2021koopman,fackeldey2020approximative,kaiser2021data,ma2019optimal}. However, one of the fundamental challenges with the current approaches to using Koopman theory for control is the bilinear nature of the Koopman-based lifting of control dynamical system. The bilinear lifting is one of the main hurdles in extending linear system tools as they inhibit the development of convex or linear methods for nonlinear control. This is in contrast to the convex framework for control design using the Perron-Frobenius operator, dual to the Koopman operator discovered in \cite{huang2022convex,yu2022data,moyalan2021sum,vaidya2023data}. {\color{black} In \cite{villanueva2021towards}, the authors have proposed the Koopman-based lifting of the Hamiltonian dynamical system arising from the Pontryagin maximum principle. One of the main focuses of \cite{villanueva2021towards} is to prove the symplectic structure of the lifted Hamiltonian system in the function space for optimal control design. The Hamiltonian system is also the main focus of this paper. However, unlike \cite{villanueva2021towards}, we discover a relationship between the spectral properties of the Koopman operator and the state space geometry of the Hamiltonian system in the form of Lagrangian submanifold for optimal control design. Unlike \cite{villanueva2021towards}, we do not rely on the infinite-dimensional linear lifting of the Hamiltonian system but propose the use of principal eigenfunctions of the Koopman operator for optimal control design. Furthermore, the use of principal eigenfunctions of the Koopman operator to discover the integrable structure of the Hamiltonian system is one of the novel contributions of this paper.}\\
\noindent {\it Main Contributions}: The main contributions of this paper are as follows. We provide two procedures for constructing the Lagrangian submanifold and the HJ solution based on spectral analysis of the Koopman operator. In our first procedure, we show that the Koopman eigenfunctions of the uncontrolled system can be used to decompose the Hamiltonian associated with the HJ equation into integrable and non-integrable parts. The integrable part of the Hamiltonian system can be solved exactly. However, for the non-integrable part, we make certain approximations leading to the approximate solution of the HJ equation. \\
The decomposition of the Hamiltonian dynamical system into integrable and non-integrable parts using the first integral of motion is proposed in \cite{sakamoto2008analytical} for approximating the HJ solution. However, unlike \cite{sakamoto2008analytical}, we present a systematic convex optimization-based approach with rigorous results on the convergence analysis for the approximation of Koopman eigenfunctions and Lagrangian submanifold. Furthermore, unlike \cite{sakamoto2008analytical}, the approximation of the Lagrangian submanifold and the HJ solution we obtained are time-independent. \\
Our second procedure for approximating the Lagrangian submanifold and the HJ solution uses the Koopman eigenfunctions of the Hamiltonian dynamical system associated with the HJ equation. In particular, the zero-level curves of the Koopman eigenfunctions are used to determine the Lagrangian submanifolds. \\
The second procedure involves computing the Koopman eigenfunction of a higher, $2n$-dimensional Hamiltonian system compared to our first procedure involving $n$-dimensional uncontrolled system. We show that the Riccatti solution corresponding to the linearized HJ equation is obtained from both methods as a particular case of the approximate HJ solution. This specific case is when only linear basis functions are used to approximate the Koopman eigenfunctions. Hence, the proposed Koopman-based approach for analyzing the HJ equation provides a natural extension of the linear system results to nonlinear systems. Finally, we demonstrate the application of the developed framework to optimal control problems.  
This paper is an extended version of \cite{vaidyacdc2022}. In particular, the first procedure based on the decomposition of Hamiltonian into an integrable and non-integrable part is new to this paper. Similarly, the results presented in this paper for procedure two are stronger than those presented in \cite{vaidyacdc2022}. Finally, rigorous convergence analysis for approximating the Koopman eigenfunctions is new to this paper. 

\section{Preliminaries and Notations}\label{section_prelim}
In this section, we present some preliminaries on the Hamiltonian dynamical system, the HJ equation, and the spectral theory of the Koopman operator. The preliminaries will also establish a connection between the Hamiltonian dynamics-based symplectic geometry framework and the solution of the HJ equation. We refer the readers to \cite{van_der_shaft_book,Lasota,abraham2008foundations,mezic2020spectrum} for further details on the preliminaries.\\

\noindent {\bf Notations}: $\mR^n$ denotes the $n$ dimensional Euclidean space. We denote by  ${\cal C}^k$ the space of $k$-times continuously differentiable functions and ${\cal C}^0$ the space of continuous function. $\mC$ denotes the set of complex numbers. 
We denote by $\bs_t(\bz)$ and $\bs_t(\bx)$ the solutions of  systems, $\dot \bz=\bF(\bz)$ and $\dot \bx=\bff(\bx)$, at time $t$ with initial condition $\bz$ and $\bx$ respectively.  

\subsection{Lagrangian Submanifold} \label{section_lagsubmanifold} 

Let $\bx=(x_1,\ldots,x_n)^\top\in \cM\subseteq \mR^n$ an $n$-dimensional space and $(\bx,\bp)=(x_1,\ldots, x_n,p_1,\ldots,p_n)$ as the cotangent bundle $T^\star\cM$. The Hamiltonian vector field is defined using a Hamiltonian function $H: T^\star \cM\to \mR$ as follows. 
\begin{align}
\dot \bx=&\frac{\partial H(\bx,\bp)}{\partial \bp}\nonumber\\
\dot \bp=&-\frac{\partial H(\bx,\bp)}{\partial \bx}.\label{hamiltonsystem}
\end{align}
Let $(\bx,\bp)=(0,0)$ be the equilibrium point of the Hamiltonian. From the property of the Hamiltonian dynamical system it follows that the equilibrium at the origin is either elliptic-type (i.e., all the eigenvalues of the linearization on the imaginary axis) or saddle-type (with eigenvalues forming a mirror image along the imaginary axis). The Hamiltonian dynamical system that arises in the context of the HJ equation, the equilibrium point at the origin is of saddle-type and hence we assume that the origin is a saddle equilibrium point \cite{van_der_shaft_book}. 

Before providing a formal definition of the Lagrangian submanifold, we provide an informal definition of Lagrangian submanifold and its connection to the HJ solution. 
The $2n$-dimensional Hamiltonian system (\ref{hamiltonsystem}) has $n$-dimensional stable (unstable), $\cM_{s(u)}$ manifold associated with the saddle-type equilibrium point at the origin. The Lagrangian submanifolds, $\cL$, are subsets of these invariant manifolds, which can be written as $\cL=\{(\bx,\bp)\in \cM_{u(s)}: (\bx,\bp=\frac{\partial V}{\partial \bx}^\top)\}$ for some scalar-valued function $V:\cM\to \mR$ i.e., these manifolds can be parameterized in terms of $\bx$ variable only. In Section \ref{section Main results}, we show the connection between the HJ equation and the Hamiltonian dynamical system (\ref{hamiltonsystem}). It is known that the scalar value function $V$ used in defining the Lagrangian submanifold $\cL$ will qualify as the solution of the HJ equation.


For a more formal definition of the Lagrangian submanifold, it is necessary to involve concepts from differential geometry. The state space of the Hamiltonian system (\ref{hamiltonsystem}) i.e., $T^\star\cM$ is a symplectic manifold with symplectic 2-form given by $\omega=\sum_i^n dp_i\wedge x_i $ (Refer to \cite{arnold2012geometrical,abraham2008foundations}) for more details on symplectic manifolds, definition of $2$-form, and the wedge product $\wedge $. 
\begin{definition}[Lagrangian submanifold]\label{def_lagrangianmanifold}
 An n-dimensional submanifold $\cL$ of $T^\star \cM$ is Lagrangian if $\omega$
 restricted to $\cL$ is zero.
\end{definition}
Now consider any $\cC^2$ function $V : \cM \to\mR$, and the $n$-dimensional submanifold
$\cL_V \subset T^\star \cM$, in local coordinates given as
\begin{align}
\cL_V=\left\{(\bx,\bp)\in T^\star \cM:\;\;\bp-\frac{\partial V}{\partial x}^\top=0\right\}.\label{lag_par}
\end{align}
It follows that $\cL_V$ is a Lagrangian submanifold as the $2$-form $\omega$ restricted to $\cL_V$ is zero. Note that Lagrangian submanifold $\cL_V$ in (\ref{lag_par}) is parameterized by $\bx$ coordinates only.   The converse of the above statement is also true \cite[Proposition 11.1.2] {van_der_shaft_book}. The Lagrangian submanifold can be obtained as 
an invariant manifold of the Hamiltonian system (\ref{hamiltonsystem}). A submanifold $\cN\subset T^\star \cM$ is invariant manifold of (\ref{hamiltonsystem}) if solutions starting on $\cN$ remains in $\cN$. We have the following Proposition from \cite{van_der_shaft_book}.    
\begin{proposition}\label{proposition_lagsubmanifold}
Let $S: \cM\to \mR$ and consider the submanifold $\cL_S\subset T^\star\cM$ of the form (\ref{lag_par}). Then
\begin{align}
H\left(\bx,\frac{\partial S}{\partial \bx}^\top\right)={\rm constant},\;\;\forall \bx\in \cM, \label{hh}
\end{align}
if and only if $\cL_S$ is an invariant submanifold of the Hamiltonian system (\ref{hamiltonsystem}).
\end{proposition}

\subsection{Spectral theory of Koopman operator}\label{section_Koopmanspectrum}
In this section, we provide a brief overview of existing results on the spectral theory of the Koopman operator. For more details on this topic, refer to \cite{mezic2020spectrum,mezic2021koopman}. Consider the dynamical system
\begin{align}
    \dot \bz={\bf F}(\bz),\label{odesys}
\end{align}
defined on a state space ${\cal Z}\subseteq \mR^p$.  The vector field $\bF$ is assumed to be smooth function. Let $\cF\subseteq \cC^0$ be the function space of observable $\psi: \cZ\to \mC$.
We have following definitions for the Koopman operator and its spectrum. 
\begin{definition}[Koopman Operator]
{\color{black} 
The family of Koopman operators $\mathbb{U}_t:\cF\to \cF$ corresponding to~\eqref{odesys} is defined as
\begin{equation}
[\mathbb{U}_t \psi](\bz) = \psi(\bs_t(\bz)). \label{koopman_operator}
\end{equation}
If in addition $\psi$ is continuously differentiable, then $f(\bz,t):=[\mathbb{U}_t \psi ](\bz)$ satisfies a partial differential equation~\cite{Lasota}:
\begin{align}
\frac{\partial f}{\partial t} = \frac{\partial f}{\partial \bz} \bF := \cK_\bF f \label{Koopmanpde}
\end{align}
with the initial condition $f(\bz,0)=\psi(\bz)$. The operator $\cK_\bF$ is the infinitesimal generator of $\mathbb{U}_t$, i.e.,
\begin{equation}
{\cal K}_{\bf F} \psi = \lim_{t \to 0} \frac{(\mathbb{U}_t - I)\psi}{t}. \label{K_generator}
\end{equation}
}
\end{definition}
It is easy to check that each $\mU_t$ is a linear operator on the space of functions, $\cF$.  
\begin{definition}\label{definition_koopmanspectrum}[Eigenvalues and Eigenfunctions of Koopman] A function $\psi_\lambda(\bz)$, assumed to be at least $\cC^1$,  is said to be an eigenfunction of the Koopman operator associated with eigenvalue $\lambda$ if
\begin{eqnarray}
[\mU_t \psi_\lambda](\bz)=e^{\lambda t}\psi_\lambda(\bz)\label{eig_koopman}.
\end{eqnarray}
Using the Koopman generator, the (\ref{eig_koopman}) can be written as 
\begin{align}
    \frac{\partial \psi_\lambda}{\partial \bz}{\bf F}=\lambda \psi_\lambda\label{eig_koopmang}.
\end{align}
\end{definition}
The eigenfunctions and eigenvalues of the Koopman operator enjoy the following property \cite{mezic2020spectrum,budivsic2012applied}. 
The spectrum of the Koopman operator, in general, is very complex.
Furthermore, the spectrum depends on the underlying functional space used in the approximation \cite{mezic2020spectrum}.
In this paper, we are interested in approximating the eigenfunctions of the Koopman operator with associated eigenvalues, the same as that of the linearization of the nonlinear system at the equilibrium point. With the hyperbolicity assumption on the equilibrium point of the system (\ref{odesys}), this part of the spectrum of interest to us is  well-defined. In the following discussion, we summarize the results from \cite{mezic2020spectrum} relevant to this paper and justify some of the claims made above on the spectrum of the Koopman operator. 

Equations (\ref{eig_koopman}) and (\ref{eig_koopmang}) provide a general definition of the Koopman spectrum. However, the spectrum can be defined over finite time or over a subset of the state space. The spectrum of interest to us in this paper could be well-defined over the subset of the state space. 
\begin{definition}[Open Eigenfunction \cite{mezic2020spectrum}]\label{definition_openeigenfunction}
Let $\psi_\lambda: \bC\to \mC$, where $\bC\subset \cZ$ is not an invariant set. Let $\bz\in  \bC$, and
$\tau \in (\tau^-(\bz),\tau^+(\bz))= I_\bz$, a connected open interval such that $\tau (\bx) \in \bC$ for all  $\tau \in I_\bz$.
If
\begin{align}[\mU_\tau \psi_\lambda](\bz) = \psi_\lambda(\bs_\tau(\bz)) =e^{\lambda \tau}  \psi_\lambda (\bz),\;\;\;\;\forall \tau \in I_\bz. 
\end{align}
Then $\psi_\lambda(\bz)$ is called the open eigenfunction of the Koopman operator family $\mU_t$, for $t\in \mR$ with eigenvalue $\lambda$. 
\end{definition}
\noindent 1. If $\bC$ is a proper invariant subset of $\cZ$ in which case $I_\bz=\mR$ for every $\bz\in \bC$, then $\psi_\lambda$ is called the subdomain eigenfunction. If $\bC=\cZ$ then $\psi_\lambda$ will be the ordinary eigenfunction associated with eigenvalue $\lambda$ as defined in (\ref{eig_koopman}). \\

\noindent 2. The open eigenfunctions as defined above can be extended from $\bC$ to a larger reachable set when $\bC$ is open based on the construction procedure outlined in  \cite[Definition 5.2, Lemma 5.1]{mezic2020spectrum}. Let $\cP_\bz$ be that larger domain.\\

\noindent 3. The eigenvalues of the linearization of the system dynamics at the origin, i.e., $\bE$, will form the eigenvalues of the Koopman operator \cite[Proposition 5.8]{mezic2020spectrum}. Our interest will be in constructing the corresponding eigenfunctions, defined over the domain $\cP_\bz$. We will refer to these eigenfunctions as {\it principal eigenfunctions} \cite{mezic2020spectrum}. \\

{\color{black}
\noindent 4. When the matrix $\bE$ has multiple eigenvalues at $\lambda$ with algebraic multiplicity not equal to geometric multiplicity, one can define {\it generalized principal eigenfunctions}. For example let $\lambda$ be the eigenvalue with algebraic multiplicity $m$ and geometric multiplicity one then generalized principal eigenfunctions, $\psi_\lambda^k(\bz)$ for $k=1,\ldots,m$ with eigenvalue $\lambda$ will satisfy
\begin{align}
[\mU_\tau \psi_\lambda^1](\bz)=e^{\lambda \tau} \psi^1_\lambda(\bz),\;\;[\mU_\tau \psi_\lambda^k](\bz)=e^{\lambda \tau} \psi^k_\lambda(\bz)+te^{\lambda \tau}\psi^{k-1}_\lambda\label{generalized_eigenfunction}
\end{align}
expressed in the differential form as 
\begin{align}\frac{\partial \psi_\lambda^1}{\partial \bz}\bF=\lambda \psi_\lambda^1,\;\;\frac{\partial \psi_\lambda^k}{\partial \bz}\bF=\lambda \psi_\lambda^k+\psi_\lambda^{k-1}\label{generalized_eigenfunction_differential}
\end{align}
for $k=2,\ldots,m$.
}\\

\noindent 5. The principal eigenfunctions can be used as a change of coordinates in the linear representation of a nonlinear system and draw a connection to the famous Hartman-Grobman theorem  on linearization and Poincare normal form \cite{arnold2012geometrical}. 
The principal eigenfunctions will be defined over a proper subset $\cP_\bz$ of the state space $\cZ$ (called subdomain eigenfunctions) or over the entire $\cZ$ \cite[Lemma 5.1, Corollary 5.1, 5.2, and 5.8]{mezic2020spectrum}. \\

The spectrum of the Koopman operator reveals essential information about the state space geometry of the dynamical system \cite{mezic2020spectrum, mezic2021koopman}. In particular, we have the following results.

\begin{corollary}\cite[Corollary 5.10] {mezic2020spectrum}\label{proposition_mainfolds}
Let the origin be the hyperbolic equilibrium point of the system (\ref{odesys}) with $\lambda_1,\ldots, \lambda_p$ the eigenvalues of the linearization of the system (\ref{odesys}) at the origin. Let  $\psi_u=\{\psi_{\lambda_1},\ldots, \psi_{\lambda_u}\}$ be open eigenfunctions associated with eigenvalues $\lambda_1,\ldots,\lambda_u$  with positive real part  and $\psi_s=\{\psi_{\lambda_{u+1}},\ldots, \psi_{\lambda_p}\}$ be open eigenfunctions associated with eigenvalues $\lambda_{u+1},\ldots,\lambda_p$ with negative real part defined over the domain $\cP_\bz$. Then, the joint level set of the (generalized) eigenfunctions  
\begin{align}
{\cal M}_{\cP}^s=\{\bz\in \cZ: \psi_{\lambda_1}(\bz)=\ldots=\psi_{\lambda_u}(\bz)=0\},    
\end{align}
forms the stable manifold on $\cP_\bz$ and the joint level set of the (generalized) eigenfunctions
\begin{align}
{\cal M}_\cP^u=\{\bz\in \cZ: \psi_{\lambda_{u+1}}(\bz)=\ldots=\psi_{\lambda_p}(\bz)=0\},
\end{align}
is the unstable manifold on $\cP_\bz$ of origin equilibrium point. 
\end{corollary}
For the stable and unstable manifolds to be well-defined it is assumed that the Jacobian matrix of unstable and stable eigenfunctions satisfy the rank conditions i.e., ${\rm rank}\left(\frac{\partial \psi_u(\bz)}{\partial \bz}\right)=u$ and ${\rm rank}\left(\frac{\partial \psi_s(\bz)}{\partial \bz}\right)=p-u$ for all $\bz\in \psi_u^{-1}(0)$ and $\bz\in \psi_s^{-1}(0)$ respectively.   
 
\section{Main Results}\label{section Main results}
The results developed in this paper can be applied to approximate the solution of the HJ equation that arise in various application such as dissipativity theory or $\cL_2$ gain analysis \cite{van_der_shaft_book}. However for concreteness, we focus specifically on the HJ equation that arises in the optimal control problem. Consider the infinite-horizon optimal control problem for the control-affine dynamical system.
\begin{align}
&\min_{\bu}\int_0^\infty q(\bx(t))+\frac{1}{2}\bu^\top(t)  \bD \bu(t) dt\label{cost}\\
&{\rm s.t.}\;\;\dot \bx=\bff(\bx)+\bg(\bx)\bu=:\bS(\bx,\bu),\label{cont_sys}
\end{align}
where $\bx\in \mR^n$ and $\bu\in \mR^p$ are the states and control input respectively. $\bx(t)$ is the trajectory of the control system with initial condition $\bx$. We make the following assumption on the system dynamics. 

\begin{assumption}\label{assume_system} 
\begin{enumerate}
\item 
{\color{black}
We assume that $\bff:\mR^n\to \mR^n$ and $\bg=(\bg_1,\ldots, \bg_p)$ with $\bg_i:\mR^n\to \mR^n$ for $i=1,\ldots p$ are $\cC^\infty$ functions of $\bx$ and  $\bff(0)=0$. $\bD\in \mR^{p\times p}$ is symmetric and positive definite matrix. Furthermore, $\bA=\frac{\partial \bff}{\partial \bx}(0)$, $q(0)=0,\frac{\partial q}{\partial \bx}(0)=0$, and $\frac{\partial^2 q}{\partial \bx^2}(0)=\bQ_0\in \mR^{n\times n}$ is a symmetric matrix.
}

\item We assume that the optimal control exists and that the optimal cost function is finite \cite{seierstad1977sufficient}. 
\end{enumerate}
\end{assumption}
 


The connection between the HJ solution and the Lagrangian submanifold can be understood based on the two approaches available for solving the optimal control problem. The first approach is based on solving the optimal control problem in the state space leading to the HJ equation, and the second approach is in the time domain based on the Pontryagin maximum principle. Let $V^\star(\bx)$ be the optimal value function assumed to be atleast $\cC^2$ function of $\bx$ i.e., 
\begin{align}
V^\star(\bx)=\min_{\bu} \int_0^\infty q(\bx(t))+\frac{1}{2}\bu^\top(t) \bD\bu(t) dt,
\end{align}
where $\bx(t)$ is the trajectory of the control system starting from initial condition $\bx$. The optimal cost function is independent of time since the cost is evaluated over an infinite time horizon. It is well known that the optimal cost function $V^\star(\bx)$ satisfy following HJ  equation \cite{van_der_shaft_book} 
\begin{align}
\frac{\partial V}{\partial \bx}{\bf f}-\frac{1}{2}\frac{\partial V}{\partial \bx}\bg \bD^{-1}\bg^\top \frac{\partial V}{\partial \bx}^\top   +q=0\label{HJ_equation},
\end{align}
and the optimal control input is expressed using the $V^\star(\bx)$ as 
\begin{align}
u^\star=- {\color{black}\bD^{-1}}\bg^\top\frac{\partial V}{\partial \bx}^\top.\label{optimal_input}
\end{align}
The HJ equation is a nonlinear partial differential equation, and it has to be solved for the unknown function $V: \mR^n\to \mR$.

An alternate approach for solving the optimal control problem is via the Pontryagin maximum principle. The optimal control problem is an infinite-dimensional optimization problem with the cost function given by (\ref{cost}) and system dynamics (\ref{cont_sys}) as constraints. For solving this optimization problem, an auxiliary function or Hamiltonian is introduced using a Lagrangian multiplier or co-state variable $\bp\in \mR^n$ as follows.
\begin{align} \bar H(\bx,\bp,\bu)=\left(\bff(\bx)+\bg(\bx)\bu\right)^\top\bp+ q(\bx)+\frac{1}{2}\bu^\top \bD \bu.\label{auxilaryfun}
\end{align}
The optimal $\bu^\star$ is obtained by taking the extremum of the Hamiltonian w.r.t. $\bu$ as  $\bu^\star(\bx)=-\bD^{-1}\bg(\bx)^\top\bp$. 
Substituting the optimal value  $\bu^\star$  in (\ref{auxilaryfun}), we obtain
\begin{align}
H(\bx,\bp)=\bff(\bx)^\top \bp -\frac{1}{2}\bp^\top \bg(\bx)\bD^{-1}\bg(\bx)^\top\bp+q(\bx).\label{hamiltonian}
\end{align}
Associated with the Hamiltonian (\ref{hamiltonian}) is a Hamiltonian dynamical system written as 
\begin{align}
 \dot \bx&=\frac{\partial H(\bx,\bp)}{\partial \bp}=\bff-\bg \bD^{-1}\bg^\top \bp\label{Ham_system}\\
    \dot \bp&=-\frac{\partial H(\bx,\bp)}{\partial \bx}=-\left(\frac{\partial \bff}{\partial \bx}\right)^\top \bp+\frac{1}{2}\frac{\partial (\bp^\top \bg \bD^{-1} \bg^\top \bp)}{\partial \bx}-\frac{\partial q}{\partial \bx}^\top.\nonumber
\end{align}

The solutions of the HJ equation (\ref{HJ_equation}) and the Hamiltonian system (\ref{Ham_system}) are closely connected. Note the similarity between the HJ equation (\ref{HJ_equation}) and the Hamiltonian (\ref{hamiltonian}), where (\ref{hamiltonian}) can be obtained from (\ref{HJ_equation}) by replacing $\frac{\partial V}{\partial \bx}^\top $ by $\bp$. The Lagrangian submanifold is the connecting link between the two approaches 
for solving the optimal control problem. \\

Following the discussion from the preliminary Section \ref{section_lagsubmanifold} on Lagrangian submanifold, we can replace the Hamiltonian and the Hamiltonian vector field (\ref{hamiltonsystem}) with the specific Hamiltonian and the Hamiltonian vector field that arise in the context of the optimal control problem i.e., (\ref{hamiltonian}) and (\ref{Ham_system})

Adding a constant to the Hamiltonian (\ref{hamiltonian}) does not change the vector field (\ref{Ham_system}) and hence Eq. (\ref{hh}) from Proposition \ref{proposition_lagsubmanifold} reduces to (\ref{HJ_equation}), thereby providing a link between the Lagrangian submanifold, HJ solution, and Hamiltonian dynamics. {\color{black}The Lagrangian submanifold can be constructed as an invariant manifold of the Hamiltonian system, but not every invariant manifold is Lagrangian.
The Lagrangian manifold of interest in the optimal control problem and used in constructing the solution of the HJ equation is uniquely obtained from the stable invariant manifold of the Hamiltonian dynamical system \ref{hamiltonsystem} \cite[Proposition 11.1.4]{van_der_shaft_book}. 
}
\begin{remark} 
\begin{enumerate}
\item There are different notions of solution to the HJ equation: classical, weak, and generalized or viscosity solution. All these solutions can be connected to the Lagrangian submanifold. Refer to \cite{day1998lagrange} on the connection between Lagrangian submanifold and different notions of solutions. The classical and weak solutions differ by their differentiability properties; the viscosity solution is non-smooth.
\item In this paper, we restrict to approximating the classical solution of the HJ equation. Hence, there is an implicit assumption that the classical solution exists to the HJ equation. 
\end{enumerate}
\end{remark}

\subsubsection{Riccatti Equation}
The Riccatti equation can be viewed as the linearization of the nonlinear HJ equation; 

\begin{align}
    \bA^\top \bP_r+\bP_r \bA-\bP_r \bR_0\bP_r+\bQ_0=0,\label{riccati_equation}
\end{align}
with $\bA=\frac{\partial \bff}{\partial \bx}(0),\;\bR_0:=\bg(0)\bD^{-1}\bg(0)^\top$ and $\bQ_0=\frac{\partial^2 q}{\partial \bx^2}(0)$. A symmetric matrix $\bP_r$ is called the stabilizing solution of  (\ref{riccati_equation}) if it is the solution of (\ref{riccati_equation}) and $\bA-\bR_0\bP_r$ is stable. 
Just like we associate Hamiltonian  system (\ref{Ham_system}) with the HJ equation (\ref{HJ_equation}), we can associate linear Hamiltonian system with Hamiltonian matrix ${\cal H}_0$  as follows,
\begin{align}
    {\cal H}_0:=\begin{pmatrix}\bA&-\bR_0\\-\bQ_0&-\bA^\top\end{pmatrix}\in \mR^{2n\times 2n}\label{ham_matrix1}.
\end{align}
The Hamiltonian matrix corresponds to linearizing the nonlinear Hamiltonian system at the origin.
The necessary and sufficient conditions for the existence of stabilizing solution to the Riccatti equation (\ref{riccati_equation}) is that 
the Hamiltonian matrix has no eigenvalues on the imaginary axis, and the generalized stable eigenspace $\mE_s$ corresponding to stable eigenvalues of $\cal H$ satisfies the following condition \cite{zhou1996robust}
\begin{align}
\mE_s\oplus {\rm Im}\begin{pmatrix}0^\top&I^\top\end{pmatrix}^\top=\mR^{2n}\label{assume_complimentary} .
\end{align}
We make following assumption on the Hamiltonian matrix $\cal H$ in (\ref{ham_matrix1}).
\begin{assumption}\label{assumption_hamiltonianeigenvalues}
 We assume that the Hamiltonian matrix $\cal H$ in (\ref{ham_matrix1}) satisfy (\ref{assume_complimentary}) with none of its eigenvalues on the imaginary axis. 
\end{assumption}

 Following the discussion in Section \ref{section_lagsubmanifold}, we know that the scalar-valued function, $V$, which is used for expressing the Lagrangian submanifolds of the Hamiltonian system (\ref{Ham_system}), i.e., $\bp=\frac{\partial V}{\partial \bx}^\top$ constitute the solution of the HJ equation. For the optimal control problem, if we are only interested in determining the optimal control, then it suffices to determine the Lagrangian submanifold as the optimal input only requires knowledge of $\frac{\partial V}{\partial \bx}$ (Eq. (\ref{optimal_input})). However, in the HJ equation that arises in other applications, such as dissipativity theory, we are interested in determining the function $V(\bx)$, which acts as a storage function.
We write the HJ equation (\ref{HJ_equation}) and Hamiltonian function (\ref{hamiltonian}) in the following form
\begin{align}
\frac{\partial V}{\partial \bx}{\bf f}-\frac{1}{2}\frac{\partial V}{\partial \bx}\bR(\bx)\frac{\partial V}{\partial \bx}^\top +q=0\label{HJ_equation1},
\end{align}
\begin{align}
H(\bx,\bp)=\bff(\bx)^\top\bp -\frac{1}{2}\bp^\top\bR(\bx)\bp+q(\bx).\label{hamiltonian1}
\end{align}
Note that for the optimal control problem 
\begin{align}
\bR(\bx)=\bg(\bx) \bD^{-1}\bg(\bx)^\top.\label{Reqn}
\end{align}
The definition of $\bR(\bx)$ and $q(\bx)$ will differ based on the problem and the HJ equation under consideration, such as $\cL_2$-gain analysis or robust control problem. It is important to emphasize that the HJ equation and the Hamiltonian are nonlinear functions of $\frac{\partial V}{\partial \bx}$ and $\bp$, respectively. Hence, using  equations (\ref{HJ_equation1}) and (\ref{hamiltonian1}) directly to solve for the Lagrangian submanifold is impossible.

 The main results of this paper provides two different procedures for approximating the solution of the HJ equation based on the spectral properties of the Koopman operator associated with dynamical system $\dot \bx=\bff(\bx)$ with $n$-dimensional state space and the Hamiltonian system (\ref{hamiltonsystem}) with $2n$-dimensional state space. 
{\color{black} Our first procedure for approximating the Lagrangian submanifold and, hence, the  HJ solution follows the approximation procedure outlined in \cite{sakamoto2008analytical} with the fundamental difference that the spectral theory of the Koopman operator forms the foundation of our approximation procedure.}



Our next two subsections discuss two  procedures developed for the approximation of the Lagrangian manifolds and the HJ solution. 
\subsection{Integrable Hamiltonian System and Koopman Eigenfunctions: Procedure One}\label{section_p1}

Our first approach for approximating the Lagrangian submanifold and the HJ solution uses the Koopman operator's principal eigenfunctions corresponding to the dynamical system $\dot \bx=\bff(\bx)$. For this subsection on procedure one, we make following assumption on the vector field.{\color{black}
\begin{assumption} We assume that the equilibrium at the origin for the uncontrolled system $\dot \bx=\bff(\bx)$ is hyperbolic i.e., none of the eigenvalues of  $\bA=\frac{\partial \bff}{\partial \bx}(0)$ are on the imaginary axis. 
\end{assumption}
}
{\color{black}Since the first procedure relies on the principal (generalized) eigenfunctions of the Koopman operator associated with uncontrolled system $\dot \bx=\bff(\bx)$, the approximate HJ solution is valid in the domain of the state space where the principal (generalized) eigenfunctions are well defined, say in domain $\cP_\bx$. In particular, the principal (generalized) eigenfunctions serve as a change of coordinates to transform the nonlinear system $\dot \bx=\bff(\bx)$ to a linear system in domain $\cP_\bx$  \cite[Lemma 5.1, Corollary 5.1, 5.2, and 5.8]{mezic2020spectrum}. This rules out the possibility of other equilibrium points or attractor sets in $\cP_\bx$.}



{\color{black} The procedure relies on decomposing the Hamiltonian in Eq. (\ref{hamiltonian1}) into integrable (i.e., the part that can be resolved exactly) and non-integrable parts. This approximation procedure has similarity with procedure one developed in \cite{sakamoto2008analytical}. The fundamental difference is that we use Koopman eigenfunctions instead of {\it integral of motions} to decompose the Hamiltonian into integrable and non-integrable parts.} We write the Hamiltonian function, $H(\bx,\bp)$ from (\ref{hamiltonian1}), as the sum of the nominal Hamiltonian, $H_0$, and its perturbation, $H_1$.
\begin{align}
H(\bx,\bp)=H_0(\bx,\bp)+H_1(\bx,\bp),\label{Hamnompert}
\end{align}
where 
\[H_0(\bx,\bp)=\bp^\top {\bf f}(\bx),\;\;\;H_1(\bx,\bp)=-{\frac{1}{2}}\bp^\top \bR(\bx)\bp+q(\bx).\]

Consider the Hamiltonian dynamical system  constructed using the nominal Hamiltonian $H_0(\bx,\bp)=\bp^\top {\bf f}(\bx)$ 
\begin{eqnarray}\label{ham_lifting} {\cal X}_{H_0}\equiv\left\{ \begin{array}{ccl}
  \dot \bx&=&\frac{\partial H_0(\bx,\bp)}{\partial \bp}={\bf f}(\bx)\nonumber\\
    \dot \bp&=&-\frac{\partial H_0(\bx,\bp)}{\partial \bx}=-\left(\frac{\partial {\bf f}}{\partial \bx}\right)^\top \bp.\end{array}\right.
\end{eqnarray}

\noindent {\bf Notation}: For the rest of this subsection, we use the following notations. We let $\lambda_1,\ldots,\lambda_j,\ldots, \lambda_q$ be the $q(\leq n)$ distinct eigenvalues of $\bA$ with algebraic multiplicity of $\lambda_j$ equal to $m_j$ and geometric multiplicity  one \footnote{The geometric multiplicity, say equal to two, can be treated as two geometric eigenspaces of dimension one. } for $j=1,\ldots,q$ and {\color{black}$\phi_{\lambda_j}^k(\bx)\in \cC^2$} for $k=1,\ldots, m_j$ be the principal (generalized) eigenfunctions associated with eigenvalue $\lambda_j$ assumed to be well-defined in domain $\cP_\bx$. For compact notation, we use
\begin{align}
\Phi(\bx)=(\phi_{\lambda_1}^1(\bx),\ldots, \phi_{\lambda_j}^{k_j}(\bx),\ldots, \phi_{\lambda_q}^{m_q}(\bx))^\top\label{principaleig}
\end{align}
be the vector of principal (generalized) eigenfunctions in real form and $\Lambda$ be the matrix of eigenvalues in real Jordan canonical form. 

Also for notations convenience at places, we will simply denote the principal (generalized) eigenfunctions and eigenvalues pair by $\{\phi_{\lambda_\ell},\lambda_\ell\}$ for $\ell=1,\ldots, n$ without the superscript for  the multiplicity. Since $\Phi(\bx)$ be the vector of principal (generalized) eigenfunctions with eigenmatrix $\Lambda$, we have 
\begin{align}
\frac{\partial \Phi}{\partial \bx}\bff=\Lambda \Phi,\;\;\;[\mU_t \Phi](\bx)=e^{\Lambda t}\Phi(\bx).\label{eigenfunctioncomapct}
\end{align}

We next show that the Koopman eigenfunctions corresponding to the dynamical system (\ref{odesys}) can be used to provide an integrable structure for the Hamiltonian system (\ref{ham_lifting}). 

\begin{proposition} \label{proposition_integrable} Let 
$\{\phi_{\lambda_\ell}(\bx),\lambda_\ell\}$ for $\ell=1,\ldots,n$ be the principal (generalized) eigenfunctions eigenvalues pair of the Koopman generator corresponding to the system, $\dot \bx=\bff(\bx)$. Then, $\{\phi_{\lambda_\ell}(\bx),\lambda_\ell\}$ and $\{H_0(\bx,\bp),0\}$ are the principal (generalized) eigenfunctions eigenvalues pair of the Koopman generator associated with the Hamiltonian system (\ref{ham_lifting}). 
\end{proposition}
\begin{proof}
The action of the Koopman generator on $\psi(\bx,\bp)$ is given by
\begin{align}
    {\cal K}_{{\cal X}_{H_0}}\psi(\bx,\bp)=\frac{\partial \psi}{\partial \bx}{\bf f}-\frac{\partial \psi}{\partial \bp}\left(\frac{\partial {\bf f}}{\partial \bx}\right)^\top{\bf p}.\label{kgeneratorHamiltonian}
\end{align}
Since $\{\phi_{\lambda_\ell},\lambda_\ell\}$ for $\ell=1,\ldots, n$ form the eigenfunctions, eigenvalues pair of the Koopman generator corresponding to system (\ref{ham_lifting}) and $\phi_{\lambda_k}(\bx)$ are not function of $\bp$, it follows that
\[{\cal K}_{{\cal X}_{H_0}}\phi_{\lambda_j}^k=\frac{\partial \phi_{\lambda_j}^k}{\partial \bx}\bff=\lambda_j \phi_{\lambda_j}^k+\phi_{\lambda_j}^{k-1}.\]
for $k=2,\ldots,m_j$. 
Similarly, by taking $\psi(\bx,\bp)=H_0(\bx,\bp)$ and using (\ref{kgeneratorHamiltonian}), we obtain
\begin{align}
   {\cal K}_{{\cal X}_{H_0}}H_0= \left(\frac{\partial H_0}{\partial \bx}\right) {\bf f}-\left(\frac{\partial H_0}{\partial \bp}\right) \left(\frac{\partial {\bf f}}{\partial \bx}\right)^\top \bp\nonumber\\
  =\bp^\top   \left(\frac{\partial {\bf f}}{\partial \bx}\right) {\bf f}-{\bf f}^\top \left(\frac{\partial {\bf f}}{\partial \bx}\right)^\top \bp=0,
\end{align}
i.e., eigenfunction with eigenvalue zero.\qed
\end{proof}
We aim to show that the eigenfunctions, eigenvalues pair from Proposition \ref{proposition_integrable} provide integrable structure to the nominal Hamiltonian $H_0$. Toward this goal, we want to determine the canonical change of coordinates for the unperturbed Hamiltonian system. This can be achieved using generating function, $W(\bx,t)$, that satisfies the following PDE \cite{arnold2012geometrical,abraham2008foundations}.
\begin{align}
    H_0\left(\bx,\frac{\partial W(\bx,t)}{\partial \bx}\right)+\frac{\partial W(\bx,t)}{\partial t}=0\label{HJ_nominal}.
\end{align}

\begin{proposition}\label{proposition_generating} Let  the  initial condition of the PDE \ref{HJ_nominal} {\color{black}$W_0(\bx)\in \cC^2$}  lie in the span of the principal (generalized) eigenfunctions, i.e., 
\begin{align}
    W_0(\bx)=\sum_{j=1}^q\sum_{k_j=1}^{m_j}  P_j^{k_j} \phi^{k_j}_{\lambda_j}(\bx)=\bP^\top \Phi(\bx),\label{W0}
\end{align}
for some $n$ constants $(P_1^1,\ldots, P_j^{k_j},\ldots, P_{q}^{m_q})^\top=: \bP\in \mR^n$ with $\Phi(\bx)$ as defined in (\ref{principaleig}). Then the solution $W(\bx,t)$ to PDE (\ref{HJ_nominal}) is given by
\begin{align}
W(\bx,t)&=\sum_{j=1}^p P_j^1 e^{-\lambda_j t} \phi_{\lambda_j}^1(\bx)\nonumber\\&+ \sum_{j=1}^q \sum_{k_j=2}^{m_j} P_j^{k_j} \left(e^{-\lambda_j t} \phi_{\lambda_j}^{k_j}(\bx)+te^{-\lambda_j t} \phi_{\lambda_j}^{k_j-1} (\bx)\right)\\&=\bP^\top e^{-\Lambda t}\Phi(\bx)
\end{align}
Denote $W(\bx,t)=\bar W(\bx,\bP,t)$, 
Then 
\begin{align}
    p_i=\frac{\partial \bar W(\bx,\bP,t)}{\partial x_i},\;\;X_j^{k_j}=\frac{\partial \bar W(\bx,\bP,t)}{\partial P_j^{k_j}},\label{canonicalcoordinates}
\end{align}
for $i=1,\ldots,n, j=1,\ldots, q,k_j=1,\ldots, m_j$ defines the canonical change of coordinates in which the nominal Hamiltonian dynamics (\ref{ham_lifting}) is completely integrable and hence of the form
\begin{align}
\dot X^{k_j}_j=0,\;\;\;\dot P^{k_j}_j=0,\label{integrable}
\end{align}
for $j=1,\ldots, q$ and $k_j=1,\ldots,m_j$. 
\end{proposition}

\begin{proof} Using compact notation to write  (\ref{canonicalcoordinates}) as follows,
\begin{align}
\bp=\frac{\partial \bar W(\bx,\bP,t)}{\partial \bx}^\top,\;\;\bX=\frac{\partial \bar W(\bx,\bP,t)}{\partial \bP}^\top\label{compactcanonical}
\end{align}
where $\bx=(x_1,\ldots,x_n)^\top,\bp=(p_1,\ldots,p_n)^\top$, and $\bX=(X_1^1,\ldots,X_{j}^{k_j},\ldots,X_q^{m_q})^\top$. Following the definition of nominal Hamiltonian, $H_0(\bx,\bp)$ from (\ref{Hamnompert}), the PDE (\ref{HJ_nominal}) can be written as  
\begin{align}
 \frac{\partial W}{\partial t}= - \left( \frac{\partial W}{\partial \bx}\right) {\bf f}.
\end{align}
Above is a Koopman PDE (\ref{Koopmanpde}) generated by vector field $\dot \bx=-{\bf f}(\bx)$ . Hence the solution of the PDE can be written  as
\begin{align}
    W(\bx,t)=W_0(\bs_{-t}(\bx)),
\end{align}
where $\bs_{-t}(\bx)$ is the solution of differential equation $\dot \bx=-{\bf f}(\bx)$ (i.e., time reversed flow of vector field $\dot \bx={\bf f}(\bx)$).  Above solution can also be written in terms of the Koopman operator as 
\begin{align}
    W(\bx,t)=W_0(\bs_{-t}(\bx))=[\mU_{-t} W_0](\bx),
\end{align}
where $\mU_{-t}$ is the Koopman operator associated with $\dot \bx=-{\bf f}(\bx)$. Since $\{\phi_{\lambda_\ell}(\bx),\lambda_\ell\}$ are the principal (generalized) eigenfunctions, eigenvalues pair of the Koopman generator corresponding to system $\dot \bx=\bff(\bx)$, it follows that $\{\phi_{\lambda_\ell},-\lambda_\ell\}$ are the principal (generalized) eigenfunctions and eigenvalues pair of the Koopman generator corresponding to the system $\dot \bx=-{\bf f}(\bx)$ as  $\frac{\partial \Phi}{\partial \bx}(-\bff(\bx))=-\Lambda \Phi(\bx)$ from (\ref{eigenfunctioncomapct}). 
With $W_0(\bx)$ of the form (\ref{W0}),  we have
    \begin{align}
    W(\bx,t)=[\mU_{-t} W_0](\bx)=\bP^\top [\mU_{-t}\Phi](\bx)=\bP^\top e^{-\Lambda t}\Phi(\bx).\label{solutionPDE}
    \end{align}
    We next show that the system is integrable in the new coordinates $(\bX,\bP)$ i.e., $\dot \bX=\dot \bP=0$. Using the compact notations in (\ref{compactcanonical}) and (\ref{solutionPDE}), we obtain
    \begin{align}\bp=\left(\frac{\partial \Phi}{\partial \bx}\right)^\top e^{-\Lambda^\top t}\bP,\;\;\;\bX=e^{-\Lambda t}\Phi(\bx)\label{xp_eqn}
    \end{align}
    $\dot \bP=0$ follows from the fact that $\bP$ is a vector of constant coefficient used in the expansion of initial condition $W_0(\bx)$. For $\dot \bX$, we have 
    \[\dot \bX=e^{-\Lambda t}\frac{\partial \Phi}{\partial \bx}\bff-e^{-\Lambda t}\Lambda \Phi
=0\]
    follows from (\ref{eigenfunctioncomapct}).

\qed
\end{proof}

Note that the canonical change of coordinates are invertible i.e.,
\begin{align}
\bx=\Phi^{-1}(e^{\Lambda t}\bX),\;\;\bP=e^{\Lambda^\top  t}\left(\frac{\partial \Phi}{\partial \bx}^\top\right)^{-1}\bp,\label{phiinverse}
\end{align}
 where, $\Phi^{-1}(\cdot)$ is the inverse of mapping $\Phi$, which is known to exist as the principal (generalized) eigenfunctions forms a diffeomorphism in $\cP_\bx$, mapping a nonlinear system to linear system (refer to discussion following Definition \ref{definition_openeigenfunction}) and \cite[Theorem 5.6, Proposition 5.8] {mezic2020spectrum}.
  In \cite{sakamoto2008analytical} similar integrable structure was derived using integral of motions of Hamiltonian dynamical system.
We next write the perturbed Hamiltonian $H_1(\bx,\bp)$ in terms of the new canonical variables $(\bX,\bP)$ as \begin{align}
    H_1(\bx,\bp)=H_1(\bx(\bX,t),\bp(\bX,\bP,t))=: \bar H_1(\bX,\bP,t).
\end{align}
\begin{align}
    \bar H_1(\bX,\bP,t)=-\frac{1}{2}\bP^\top e^{-\Lambda t}\left(\frac{\partial \Phi}{\partial \bx}\right)\bR(\bx)\left(\frac{\partial \Phi}{\partial \bx}\right)^\top e^{-\Lambda^\top t}\bP\nonumber\\
    +q(\Phi^{-1}(e^{\Lambda t}\bX)).\label{HamiltonianXP}
\end{align}
The Hamiltonian system in the new coordinates $(\bX,\bP)$ is given by
\begin{align}
    \dot \bX=\frac{\partial \bar H_1}{\partial \bP},\;\;\;
    \dot \bP=-\frac{\partial \bar H_1}{\partial \bX} \label{Ham_dynamics22}.
\end{align}
From the above discussion, it follows that there are no approximations involved in arriving at the Hamiltonian system (\ref{Ham_dynamics22}) from (\ref{Ham_system}).
However, the Hamiltonian system in $(\bX,\bP)$ coordinates (\ref{Ham_dynamics22}) is still nonlinear, and hence it is challenging to find the Lagrangian submanifold and, therefore, the HJ solution. We make the following approximations for the approximate computation of the Lagrangian submanifold and the HJ solution.

\begin{approximation}\label{approximation}
We simplify the Hamiltonian $\bar H_1(\bX,\bP,t)$ in (\ref{HamiltonianXP}) by making following approximation. 
\begin{align}
\frac{\partial \Phi}{\partial \bx}\bR(\bx)\frac{\partial \Phi}{\partial \bx}^\top\approx \bR_1,\;\;\&\;\;q(\bx)\approx \frac{1}{2}\Phi^\top \bQ_1\Phi,
\end{align}
where $\bR_1$ and $\bQ_1$ are some constant positive matrices. The approximation $\frac{\partial \Phi}{\partial \bx}\bR(\bx)\frac{\partial \Phi}{\partial \bx}^\top\approx \bR_1$ can be broken down as follows. First, the matrix $\bR(\bx)=\bg(\bx)\bD^{-1}\bg(\bx)^\top\approx \bR_0$ is a constant matrix and will be true if the input matrix $\bg(\bx)=\bB$ is constant, if not then $\bR_0$ can be obtained as $\bR_0=\bg(0)\bD^{-1}\bg(0)^\top$. Second, the eigenfunctions $\Phi(\bx)$ which admits a decomposition into linear and nonlinear parts, {\color{black}i.e., $\Phi(\bx)=\bV^\top \bx+\bh_1(\bx)$, where $\bV^\top\bx$ the linear part with matrix $\bV$ the generalized left eigenvectors of $\bA$ with eigenmatrix $\Lambda$ i.e., $\bV^\top \bA=\Lambda \bV^\top$ and $\bh_1(\bx)$ the purely nonlinear part and hence satisfying $\frac{\partial \bh_1}{\partial \bx}(0)=0$}
(refer to Section \ref{section_compute} for more details). Then $\Phi(\bx)$ is approximated as
\[\Phi(\bx)\approx \bV^\top \bx.\] 
Combining these, we can write
\begin{align}
&\frac{\partial \Phi}{\partial \bx}\bR(\bx)\frac{\partial \Phi}{\partial \bx}^\top=\left(\bV+\frac{\partial  \bh_1}{\partial \bx}\right)\left(\bR_0+\bR(\bx)-\bR_0\right)\left(\bV+\frac{\partial \bh_1}{\partial \bx}\right)^\top\nonumber\\
&=\bV \bR_0\bV^\top +O(|\bx|^2)\approx \bV \bR_0\bV^\top =: \bR_1.\label{assumeR1} 
\end{align}
The other approximation $q(\bx)\approx \Phi^\top(\bx) \bQ_1 \Phi(\bx)$ can be understood by projecting the state cost along the eigenfunctions $\phi_{\lambda_i}(\bx) \phi_{\lambda_\ell}(\bx)$ for $i,\ell=1,\ldots, n$. In particular, for the case when the state cost is quadratic, i.e., $q(\bx)=\frac{1}{2}\bx^\top \bQ_0 \bx$, we have 
\begin{align}
&\frac{1}{2}\Phi^\top \bQ_1\Phi= \frac{1}{2}(\bV^\top \bx +\bh_1(\bx))^\top \bQ_1 (\bV^\top \bx +\bh_1(\bx))\nonumber\\
&= \frac{1}{2}\bx^\top \bV\bQ_1 \bV^\top \bx+O(|\bx|^3)\approx \frac{1}{2}\bx^\top \bQ_0 \bx\label{assumeQ1}
\end{align}
and where $\bQ_1$ is defined as $\bQ_1:=\bV^{-1}\bQ_0 (\bV^\top)^{-1}$.
\end{approximation}

Using the above approximation,  the Hamiltonian can be written as 
\[\bar H_1(\bX,\bP,t)=-\frac{1}{2}\bP^\top e^{-\Lambda t}\bR_1 e^{-\Lambda^\top t}\bP+\frac{1}{2}\bX^\top e^{\Lambda^\top t}\bQ_1 e^{\Lambda t}\bX.\]
The nonlinear Hamiltonian system
(\ref{Ham_dynamics22}) get transformed to following linear Hamiltonian system.
\begin{align}
    \begin{pmatrix}\dot {\bX}\\\dot {\bP}\end{pmatrix}=\begin{pmatrix}0&-e^{-\Lambda t} \bR_1 e^{-\Lambda^\top t}\\-e^{\Lambda^\top t}\bQ_1 e^{\Lambda t}&0\end{pmatrix}\begin{pmatrix}\bX\\\bP\end{pmatrix}.\label{ham_linear}
\end{align}

The time-varying Hamiltonian system can be converted to time-invariant system by performing one more linear time-varying  change of coordinates as \begin{align}\bar \bX=e^{\Lambda t}\bX,\;\;\;\bar \bP=e^{-\Lambda^\top t}\bP\label{coordinatechange}.
\end{align}
Writing (\ref{ham_linear}) in $(\bar \bX, \bar \bP)$ coordinates, we obtain

\begin{align}
   \begin{pmatrix}\dot{\bar \bX}\\\dot {\bar \bP}\end{pmatrix}= \begin{pmatrix}\Lambda&-  \bR_1\\-\bQ_1&-\Lambda^\top\end{pmatrix} \begin{pmatrix}{\bar \bX}\\{\bar \bP}\end{pmatrix}=:{\cal H}_1\begin{pmatrix}{\bar \bX}\\{\bar \bP}\end{pmatrix}\label{ham_matrix}.
\end{align}
Since the system equations are linear, an analytical expression for the Lagrangian subspace and the HJ solution can be found. The problem of computing the Lagrangian subspace for the linear system can be reduced to the solution of the Riccatti equation. The condition for the existence of Lagrangian submanifold and stabilizing solution to the Riccatti equation for ${\cal H}_1$ are the same as that on ${\cal H}_0$ in Eq. (\ref{ham_matrix1}) (Assumption \ref{assumption_hamiltonianeigenvalues}). Note that ${\cal H}_0$ and ${\cal H}_1$ are related by a linear change of coordinates and hence ${\cal H}_1$ satisfies Assumption \ref{assumption_hamiltonianeigenvalues}.  

The results we present next are known in the literature in different forms. However, we offer these results in the spirit of the Koopman theory, which specializes in linear systems case.
The main result of this construction procedure is summarized in the following proposition. 
\begin{proposition}\label{Proposition_Lagsubspace}For the linear Hamiltonian system (\ref{ham_matrix}) under Assumption \ref{assumption_hamiltonianeigenvalues} the Lagrangian subspace is given by
\begin{align}\hat \cL_l=\{(\bar\bX,\bar\bP): \bar\bP= \bL \bar\bX,\;\;\bL=-\bD_2^{-1}\bD_1 \},\label{lmanifold}
\end{align}
where 
\begin{align}
\bD=\begin{pmatrix}\bd_1^\top\\\vdots\\\bd_n^\top\end{pmatrix}=\begin{pmatrix}\bD_1&\bD_2\end{pmatrix}\in\mR^{n\times 2n},\;\;\bD_i\in \mR^{n\times n},\;i=1,2\label{D_def}.
\end{align}
and $\bd_j^\top$ for $j=1,\ldots, n$ are the left generalized eigenvectors associated with eigenvalues with positive real part of the Hamiltonian matrix ${\cal H}_1$. 
\end{proposition}
Refer to the Appendix for the proof. We next show that the matrix $\bL$ from Proposition \ref{Proposition_Lagsubspace} satisfies the Riccatti equation. 
\begin{proposition}\label{prop_riccati_lagsubspace}
The matrix $\bL$ from Proposition \ref{Proposition_Lagsubspace} satisfies the Riccatti equation
\begin{align}\Lambda^\top \bL+\bL \Lambda -\bL\bR_1\bL+ \bQ_1=0\label{L_riccatti},
\end{align}
and the spectrum of the closed loop linear system $\Lambda -\hat \bR \bL$ is in the left half plane.
\end{proposition}
Refer to Appendix for the proof. 
The Lagrangian submanifold can now be computed as follows. From (\ref{lmanifold}), we have  
$$\bar \bP=\bL \bar \bX \implies {\rm from \; Eq. (\ref{coordinatechange})}\implies e^{-\Lambda^\top t}\bP=\bL e^{\Lambda t}\bX.$$
Pre-multiplying the above by $\frac{\partial \Phi}{\partial \bx}$ and using (\ref{xp_eqn}), we obtain
\[\bp=\left(\frac{\partial \Phi}{\partial \bx}\right)e^{-\Lambda^\top t}\bP=\left(\frac{\partial \Phi}{\partial \bx}\right)\bL e^{\Lambda t}\bX=\left(\frac{\partial \Phi}{\partial \bx}\right)\bL \Phi. \]
We thus obtain the following expression for the approximate Lagrangian submanifold
\begin{align}\hat \cL=\left\{(\bx,\bp): \bp=\left(\frac{\partial \Phi}{\partial \bx}\right)^\top\bL \Phi\right\}.\label{approximate_lagmanifold}
\end{align}
We know that the Lagrangian submanifold is obtained as a gradient of a scalar value function i.e., $\bp=\frac{\partial V}{\partial \bx}$. Hence from (\ref{approximate_lagmanifold}), we obtain 
\begin{align}
V(\bx)=\frac{1}{2}\Phi(\bx)^\top \bL \Phi(\bx)\label{Procedure1V}.
\end{align} as the approximate HJ solution. This suggests that the approximated optimal cost is quadratic in Koopman eigenfunction coordinates where following the results of Proposition \ref{prop_riccati_lagsubspace}, the matrix $\bL$ is obtained as a solution of the Riccatti equation. 
\begin{remark}\label{remark_riccatti}
We observe from (\ref{approximate_lagmanifold}) that the Riccatti solution from the linearized HJ equation is obtained as a particular case of the Lagrangian submanifold construction proposed in procedure one. In particular, the special case will correspond to the principal (generalized) eigenfunctions  $\Phi(\bx)$ being approximated as a linear function of $\bx$, i.e., $\Phi(\bx)=\bV^\top \bx$ with $\bV^\top \bA=\Lambda \bV^\top$. In this case, we have 
\[V(\bx)=\frac{1}{2}\Phi^\top(\bx)\bL \Phi(\bx)=\frac{1}{2}\bx^\top \bV\bL \bV^\top \bx=\frac{1}{2}\bx^\top \bP_r\bx.\]

Since we know that the matrix $\bL$ satisfies the Riccatti equation  (\ref{L_riccatti})  it then follows that the matrix $\bP_r$ satisfies  Riccatti equation (\ref{riccati_equation}).
This proves that the  Riccatti solution is embedded in the proposed approximation of the HJ solution obtained using procedure one. 
\end{remark}
{\color{black}
\begin{remark} Extending the HJ solution from the infinite horizon to the finite horizon optimal control problem using Procedure One is relatively straightforward. In particular, for the finite time horizon problem, the HJ solution can be approximated as $V(\bx,t)=\frac{1}{2}\Phi^\top(\bx)\bL(t)\Phi(\bx)$, where $\bL(t)$ will be the solution of the time-varying Riccatti equation (\ref{L_riccatti}) with right-hand side replaced with $\dot \bL(t)$ and with appropriate constraint at terminal time $T$ on the state $\bL(T)$.  
\end{remark}
}
\noindent The procedure one can be summarized as follows.

\begin{procedure}{Procedure 1} \label{alg1}
{\color{black}
\begin{enumerate}
    \item Compute the principal (generalized) eigenfunction, $\Phi(\bx)$, for system $\dot \bx=\bff(\bx)$ with eigenmatrix $\Lambda$ in real Jordan canonical form. 

    \item With $\bR_1$ and $\bQ_1$ as given in Eqs. (\ref{assumeR1}) and (\ref{assumeQ1}) respectively, solve the following Riccatti equation for $\bL$
    \[\Lambda^\top \bL+\bL \Lambda -\bL\bR_1\bL+ \bQ_1=0\]
    
   \item The approximation of the HJ solution is given by $\frac{1}{2}\Phi^\top(\bx)\bL \Phi(\bx)$.
\end{enumerate}}
\end{procedure}
\subsection{Koopman Eigenfunctions of the Hamiltonian System : Procedure Two}\label{section_p2}
While the procedure one makes use of eigenfunctions of the uncontrolled system, $\dot \bx=\bff(\bx)$, the procedure two relies on the eigenfunctions of the Hamiltonian system for the approximation of the Lagrangian submanifold and the HJ solution. The Hamiltonian system is repeated here for convenience.   
\begin{align}
    \dot \bx&={\bf f}(\bx)-\bR(\bx)\bp\nonumber\\
    \dot \bp&=-\left(\frac{\partial {\bf f}}{\partial \bx}\right)^\top\bp+\frac{1}{2}\left(\frac{\partial \bp^\top \bR(\bx)\bp}{\partial \bx}\right)^\top-\frac{\partial q}{\partial \bx}^\top\label{Ham_dynamicsnew}.
\end{align}

{\color{black}
We write the above dynamical system compactly  as $\dot \bz=\bF(\bz)$, where $\bz=(\bx^\top,\bp^\top)^\top$ and $\bF(\bz)$ as the right hand side of (\ref{Ham_dynamicsnew})}. The linearization of the nonlinear Hamiltonian system at the origin is given by the Hamiltonian matrix ${\cal H}_0$ in Eq. (\ref{ham_matrix1}).
Following Assumption \ref{assume_system}, we know that the equilibrium point of the Hamiltonian system (\ref{Ham_dynamicsnew}) is  hyperbolic, with eigenvalues forming a mirror image along the imaginary axis. Hence, the Koopman operator associated with the $2n$-dimensional Hamiltonian system admits $2n$ principal (generalized) eigenfunctions $\Psi(\bz)\in \cC^2$ associated with the eigenmatrix $\Lambda_{\cH}$ in real Jordan canonical form. The matrix $\Lambda_\cH$ is also the eigenmatrix of $\cH_0$.  

We are only interested in computing $n$ out of the $2n$ principal (generalized) eigenfunctions. Which of the $n$  eigenfunctions to compute will depend on the HJ equation under consideration. For example, in the optimal control problem, the Lagrangian submanifold is the subset of the stable manifold \cite[Proposition 11.2.2] {van_der_shaft_book} as the optimal control is stabilizing. Hence, for the optimal control problem, we would compute the principal (generalized) eigenfunctions corresponding to the eigenvalues with positive real part, as the joint zero-level curves of the unstable eigenfunctions constitute the stable manifold (Corollary \ref{proposition_mainfolds}). 
Let $\Psi_u(\bz):\mR^{2n}\to \mR^n$ be the principal (generalized) eigenfunctions corresponding to the eigenvalues with positive real part of $\cH_0$. Unlike procedure one where the eigenfunctions of the Koopman operator for uncontrolled system $\dot \bx=\bff(\bx)$ were only defined in the domain $\cP_\bx\subseteq \mR^n$ containing the origin,  $\Psi(\bz)$ are well defined globally as the optimal control is defined globally. However, for the purpose of computation, we restrict the computation domain of $\Psi(\bz)$ to be $\cP_\bz$. The stable manifold is then defined using the joint-zero level set of principal (generalized) eigenfunctions as
\begin{align}\cM_s=\{\bz\in \cP_\bz : \Psi_u(\bz)=0\}.\label{zerolevelcurve}
\end{align}
The Lagrangian submanifold is obtained from the stable manifold as a graph of the gradient of the scalar function $V(\bx)$, i.e.,
\begin{align}
\cL=\left\{\bz\in \cM_s: \left(\bx,\bp=\frac{\partial V}{\partial \bx}^\top\right)\right\}.\label{zerolevelLagrangian}
\end{align}
The scalar valued function, $V(\bx)$, will be the solution of the HJ equation. 
\begin{remark}\label{remark_basischoice} Our ultimate objective is to approximate the Lagrangian submanifold, $\cL$, in Eq. (\ref{zerolevelLagrangian}) and the HJ solution, $V$, but not the principal (generalized) eigenfunctions, $\Psi_u$, of the Hamiltonian system per se. The principal (generalized) eigenfunctions of the Hamiltonian system are used as an intermediate step towards approximation of the HJ solution. With this ultimate goal in mind, we observe that the 
 Lagrangian submanifold is a subset of stable invariant manifold characterized by a set of all $(\bx,\bp) $ such that $\bp-\frac{\partial V}{\partial \bx}^\top =0$, i.e., linear function of $\bp$. This linear paramaterization of the Lagrangian submanifold in variable $\bp$ is crucial in selecting basis functions for approximating the Koopman eigenfunctions. In particular, we choose the basis functions  to be linear in the variable $\bp$ for the approximation of the principal (generalized)  eigenfunctions.
\end{remark}
Our main contribution towards approximating the Koopman principal (generalized) eigenfunctions is presented in Section \ref{section_compute}. In this section, we will assume that the principal (generalized) eigenfunctions are approximated and demonstrate their application for the computation of the Lagrangian submanifold and the HJ solution. Towards this goal, we first write the Hamiltonian system (\ref{Ham_dynamicsnew})  and split it into linear and nonlinear parts as 
\begin{align}
\dot \bz=\bF(\bz)={\cal H}_0 \bz+\bF_n(\bz),
\end{align}
where, $\bz=(\bx^\top,\bp^\top)^\top$, ${\cal H}_0$ is given in (\ref{ham_matrix1}) and $\bF_n(\bz)=\bF(\bz)-{\cal H}_0\bz$. We now introduce a few notations to facilitate the following discussion. 
Let $\bW^\top$ consists of the left (generalized) eigenvectors of the Hamiltonian matrix ${\cal H}_0$ with eigenmatrix $\Lambda_\cH$  in real Jordan canonocial form i.e., $\bW^\top{\cal H}_1=\Lambda_\cH \bW^\top$. The matrix $\bW^\top \in \mR^{2n\times 2n}$ admits following decomposition 
\begin{align}\bW^\top=\begin{pmatrix}\bW_u^\top\\\bW_s^\top\end{pmatrix}\in\mR^{2n\times 2n}\label{W_def},
\end{align}
where $\bW_u^\top\in \mR^{n\times 2n},\bW_s^\top\in \mR^{n\times 2n}$ corresponds to (generalized) unstable and (generalized) stable linear subspace respectively.
Following the decomposition of the Hamiltonian system into linear and nonlinear parts, the principal (generalized) eigenfunctions can also be decomposed into linear and nonlinear parts as follows. 
\begin{align}\mR^{n}\ni\Psi_u(\bz)=\bW_u^\top\begin{pmatrix}\bx^\top&\bp^\top\end{pmatrix}^\top+\bU\Gamma_M(\bz)\label{zz1},
\end{align}
where, $\bW_u^\top \bz$ is the linear part and $\bU\Gamma_M(\bz)$ is the nonlinear part. Similar decomposition also applies to principal (generalized) eigenfunctions corresponding to eigenvalue with negative real parts.
The nonlinear part of the principal (generalized) eigenfunction is approximated by the term $\bU \Gamma_M(\bz)$, where $\bU\in \mR^{n\times M}$ and $\Gamma_M(\bz): \mR^{2n}\to \mR^M$ are the finite basis functions used in the approximation. 
Our proposed computational framework for approximating the nonlinear part of the principal (generalized) eigenfunctions is presented in Section \ref{section_compute}. The main contribution and conclusion of Section \ref{section_compute} is that the matrix $\bU$ can be obtained as a solution to the least square problem using data.  
Following Remark \ref{remark_basischoice}, the basis functions are assumed to be of the form
\begin{align} \Gamma_M(\bz)=\begin{pmatrix}\Xi_1^\top(\bx) & (\Xi_2(\bx)\bp)^\top\end{pmatrix}^\top\in \mR^M\label{basis_proce2}.
 \end{align}
where $\Xi_1(\bx)\in \mR^N$ and $\Xi_2(\bx)\in \mR^{(M-N)\times n}$.
Since $\bU \Gamma_M(\bz)$ captures the pure nonlinear part of the principal (generalized) eigenfunctions, we chose basis function $\Gamma_M(\bz)$ such that $\frac{\partial \Gamma_M}{\partial \bz}(0)=0$. 
Given the form of $\bW_u$ and $\Gamma_M(\bz)$, we can decompose (\ref{zz1}) to determine the stable manifold as the joint zero-level set of principal (generalized)  eigenfunctions corresponding to eigenvalues with positive real part i.e., 
\begin{align}
   0=\Psi_u(\bx,\bp):= \bW_{u}^\top \begin{pmatrix}\bx\\\bp\end{pmatrix}+\begin{pmatrix}\bU_{11}&\bU_{12}\end{pmatrix}\begin{pmatrix}\Xi_1(\bx)\\\Xi_2(\bx)\bp\end{pmatrix}\label{stable_mm},
\end{align}
where we split the matrix $\bU$ with $\bU_{11}\in \mR^{n\times N}, \bU_{12}\in \mR^{n\times (M-N)}$. Following (\ref{zerolevelcurve}) and (\ref{zerolevelLagrangian}), the objective is to find Lagrangian submanifold of the form $\bp=\frac{\partial V}{\partial \bx}^\top$ for some unknown function $V$ such that $\Psi_u(\bx,\bp=\frac{\partial V}{\partial \bx}^\top)=0$. We seek the following form of the Lagrangian submanifold consisting of linear and nonlinear terms 
\begin{align}
   \bp= \frac{\partial V}{\partial \bx}^\top=:\bp_l+\bp_n\label{param_lnl_lag}.
\end{align} 
Substituting the assumed form of the Lagrangian submanifold from (\ref{param_lnl_lag}) in (\ref{stable_mm}), we obtain
\begin{align}&\bW_{u1}^\top \bx+\bW_{u2}^\top(\bp_l+\bp_n)+\bU_{11}\Xi_1(\bx)\nonumber\\&+\bU_{12}\Xi_2(\bx)(\bp_l+\bp_n)=0\label{zerolevelcurveeq}.
\end{align}
Equating the linear and nonlinear terms, we obtain
\begin{align*}
\bW_{u1}^\top \bx+\bW_{u2}^\top \bp_l=0\implies \bp_l:=\bJ_l\bx=-(\bW_{u2}^\top)^{-1}\bW_{u1}\bx
\end{align*}
where, the invertibility of matrix $\bW_{u2}$ follows from the fact the Hamiltonian matrix ${\cal H}_0$ satisfies Assumption \ref{assumption_hamiltonianeigenvalues} (refer to the proof of Proposition \ref{Proposition_Lagsubspace} in the Appendix). The optimal $\bJ_l^\star$ and hence $\bp_l^\star$ is given by 
\begin{align}
    \bp_l^\star=\bJ_l^\star \bx=-(\bW_{u2}^\top)^{-1}\bW_{u1}^\top \bx\label{linearterm}.
\end{align}
We notice, following results of Propositions \ref{Proposition_Lagsubspace} and \ref{prop_riccati_lagsubspace}, that $\bJ_l^\star$ matches with the solution of Riccatti equation. In particular $\bP_r:=\bV\bJ_l^\star \bV^\top $, where $\bV^\top \bA=\Lambda \bV$ satisfies the following Riccatti equation
\begin{align}\bA^\top \bP_r+\bP_r \bA-\bP_r\bR_0\bP_r+ \bQ_0=0\label{rr}.
\end{align}
Substituting the $\bp_l^\star$ in Eq. (\ref{zerolevelcurveeq})  we obtain following equation for the nonlinear part
\begin{align}
\bW_{u_2}^\top \bp_n+\bU_{11}\Xi_1(\bx)+\bU_{12}\Xi_2(\bx)\bp_l^\star+\bU_{12}\Xi_2(\bx)\bp_n=0
\label{ss}.
\end{align}
$\bp_l^\star=\bJ_l^\star\bx$ is a linear function of $\bx$ (Eq. (\ref{linearterm})), define
\begin{align}\bG_1(\bx):=\bU_{11}\Xi_1(\bx)+\bU_{12}\Xi_2(\bx)\bJ_l^\star \bx\in \mR^n\nonumber\\\bG_2(\bx):=\bW_{u_2}^\top +\bU_{12}\Xi_2(\bx)\in \mR^{n\times n}\label{defG12}.\end{align} With this definition, we can write (\ref{ss}) as
\begin{align}
\bG_1(\bx)+\bG_2(\bx)\bp_n=0\label{zeroLM}.
\end{align}
Under the assumption that $\bG_2(\bx)$ is invertible, we can obtain $\bp_n$ as
\begin{align}\bp_n^\star= -\bG_2(\bx)^{-1}\bG_1(\bx)\label{nonlinearterm}.
\end{align}
The invertibility assumption on $\bG_2(\bx)$ can be viewed as the nonlinear generalization of invertibility of matrix $\bW_{u2}$. The invertibility of $\bW_{u_2}$ follows  from Assumption \ref{assumption_hamiltonianeigenvalues} (refer to the proof of Proposition \ref{Proposition_Lagsubspace} in Appendix).
Combining (\ref{linearterm}) and (\ref{nonlinearterm}), the approximation of the Lagrangian submanifold is given by
\begin{align}
\bp^\star=-(\bW_{u2}^\top)^{-1}\bW_{u1}^\top \bx- \bG_2(\bx)^{-1}\bG_1(\bx) \label{LMproce2}.
\end{align}
There is a  similarity between the formula for linear $\bp_l$ in Eq. 
 (\ref{linearterm}) and nonlinear, $\bp_n$ in Eq. (\ref{nonlinearterm}). Both these formula has same functional form where $\bG_2(\bx)$ and $\bG_1(\bx)$ can be viewed as nonlinear generalization of $\bW_{u_2}$ and $\bW_{u_1}$ respectively. It suffices to use (\ref{LMproce2}) if we are only interested in computing the optimal control. Alternatively, if we are interested in computing the optimal cost, we proceed as follows. Let $\Xi_3(\bx)\in \mR^{M_1}$ be the basis function satisfying $\frac{\partial \Xi_3(0)}{\partial \bx}=0$ and $V_n(\bx)=\frac{1}{2}\Xi_3^\top(\bx) \bJ_n \Xi_3(\bx)$, for some positive $\bJ_n\in \mR^{M_1\times M_1}$, be the parametrization of the optimal cost function corresponding to terms containing higher than quadratic nonlinearity. Then substituting for $\bp_n=\frac{\partial V_n}{\partial \bx}^\top=\frac{\partial \Xi_{3}}{\partial \bx}^\top \bJ_n \Xi_3(\bx)=\bar \Xi_3 {\rm vec}({\bJ_n})$ in Eq. (\ref{zeroLM}). The ${\rm vec}(\bJ_n)\in \mR^{\frac{M_1(M_1+1)}{2}}$ is the vector form of the matrix $\bJ_n$ and $\bar \Xi_3(\bx)\in \mR^{n\times M_1}$ is obtained from $\frac{\partial \Xi_3}{\partial \bx}$ and $\Xi_3(\bx)$. Using (\ref{zeroLM}), we can write the following optimization problem to solve for $\bJ_n$.
\begin{align}
\min_{\bJ_n\geq 0}\|{\cal B}{\rm vec}(\bJ_n)-\cA\|\label{J_optimization},
\end{align}
{\small 
\[{\cal B}=\begin{pmatrix}\bar \bG_2(\bx_1)^\top,\ldots, \bar \bG_2(\bx_L)^\top\end{pmatrix}^\top,\;-{\cal A}=\begin{pmatrix}\bG_1(\bx_1)^\top,\ldots, \bG_1(\bx_L)^\top\end{pmatrix}^\top\]}
with $\bar G_2(\bx):=\bG_2(\bx)\bar \Xi_3(\bx)$ and  $\{\bx_k\}_{k=1}^L$ are the sampled data points. 
The solution of the above optimization problem, $\bJ_n^\star$, will be used to write the optimal cost function as 
\begin{align}
V(\bx)=\frac{1}{2}(\bx^\top \bJ_l^\star \bx+\Xi_3(\bx)^\top \bJ_n^\star \Xi_3(\bx)). \label{HJsolutionproc2}
\end{align}
Based on (\ref{LMproce2}), we obtain following expression for the optimal control
\begin{align}\bu^\star=-\bD^{-1}\bg^\top\left((\bW_{u2}^\top)^{-1}\bW_{u1}^\top \bx+\bG_2(\bx)^{-1}\bG_1(\bx) \right) \label{formula1_procedure2}.
\end{align}
{\color{black}
\begin{remark} The results from Section \ref{section_compute} prove that the approximated principal eigenfunctions are optimal Galerkin projections of the true principal eigenfunctions on the finite-dimensional basis function. In particular, using the results of Section \ref{section_compute}, it follows that the unstable eigenfunction and its zero-level curve (i.e., stable manifold) in  Eq. (\ref{stable_mm}) and (\ref{zerolevelcurveeq}) are optimal Galerkin projection of the true stable manifold on the finite-dimensional basis spanned by $\Gamma_M(\bz)$. The Lagrangian manifold is the subset of the stable manifold, hence the analytical construction of the Lagrangian submanifold and optimal control in Eqs. (\ref{LMproce2}) and (\ref{formula1_procedure2}) are the optimal Galerkin projection of true Lagrangian manifold and control input on the finite-dimensional basis spanned by $\Gamma_M(\bz)$.
\end{remark}}
\noindent Our approach for the construction of the Lagrangian submanifold can be summarized as following procedure.
\begin{procedure}{Procedure 2}
{\color{black}
\begin{enumerate}

    \item Choose a finite-dimensional basis function, $\Gamma_M(\bx,\bp)$, as given in Eq.  (\ref{basis_proce2}) for the approximation of principal (generalized) eigenfunctions, $\Psi(\bz)$, of the Koopman operator associated with the Hamiltonian system (\ref{Ham_dynamicsnew}).
    \item Construct the stable manifold of the Hamiltonian system as a joint zero-level set of the principal (generalized) eigenfunctions, $\Psi_u(\bz)$, corresponding to eigenvalues with positive real part (Eq. (\ref{stable_mm})).
    \item Use coefficient matrices used in the expansion of principal (generalized) eigenfunction to define $\bG_1$ and $\bG_2$ in Eq. (\ref{defG12}) and linear and nonlinear parts of the Lagrangian manifold in (\ref{linearterm}) and (\ref{nonlinearterm}) respectively.
    \item Use $\bG_1$, $\bG_2$, and linear manifolds to obtain the approximation of Lagrangian submanifold (\ref{LMproce2}). 
    
    

\end{enumerate}
}
\end{procedure}


\subsection{Computation of Principal Eigenfunctions: Convergence Analysis}\label{section_compute}
 There are several approaches available for the computation of the Koopman operator, including results that provide sample complexity-based error bounds in the computation \cite{housparse,klus2020eigendecompositions,korda2018convergence,korda2020optimal}. This section presents results for directly approximating the principal eigenfunctions without approximating the Koopman operator itself. We present the results for approximating the principal (generalized) eigenfunctions for system $\dot \bz=\bF(\bz)$. Also for notations convenience and simplicity of presentation, we present results for the approximation of eigenfunctions with real eigenvalues. The results for the multiple eigenvalues and complex eigenvalues will follow along similar lines. We decompose the system (\ref{odesys}) into linear and nonlinear parts as 
 \begin{align}
 \dot \bz=\bF(\bz)=\bE\bz+(\bF(\bz)-\bE \bz)=:\bE \bz+\bF_n(\bz)\label{sys_decompose}.
 \end{align} 
 Let $\lambda$ be the Koopman generator's eigenvalues and also of $\bE$. The principal (generalized) eigenfunction corresponding to eigenvalue $\lambda$ admits the decomposition into linear and nonlinear parts. 
 \begin{align}
 \psi_\lambda(\bz)=\bw^\top \bz+h(\bz) \label{eigen_decompose},
 \end{align}
 where $ \bw^\top \bz$ and $h(\bz)$ are the principal eigenfunction's linear and purely nonlinear parts, respectively. Substituting (\ref{eigen_decompose}) in (\ref{eig_koopmang}) and using (\ref{sys_decompose}), we obtain following equations to be satisfied by $ \bw$ and $h(\bz)$
\begin{align}
\bw^\top \bE=\lambda  \bw^\top,\;\;\frac{\partial h(\bz)}{\partial \bz}\bF(\bz)-\lambda h(\bz)+\bw^\top \bF_n(\bz)=0\label{linear_nonlinear_eig}.
\end{align}
So, the linear part of the principal eigenfunction can be found using the left eigenvector with eigenvalue $\lambda$ of matrix $\bE$, and the nonlinear term satisfies the linear partial differential equation. 
We make the following assumption on the basis functions used in the approximation of the nonlinear term $h(\bz)$. 
\begin{assumption}\label{assume_basisfunction} For the purpose of computation, we restrict the domain to $\cZ$ a subset of $\mR^p$. 
We assume that the basis function $\Gamma_M=\{\gamma_j\}_{j=1}^M$ with $\gamma_j: \cZ\to \mR$ are $\mu$ linearly independent i.e.,
\begin{align}
\mu\{\bz\in \cZ: \bc^\top \Gamma(\bz)=0\}=0.
\end{align}
Also with no loss of generality we assume that  $\{\frac{\partial \gamma_j}{\partial \bx}\bF-\lambda \gamma_j\}_{j=1}^M$ for any given eigenvalue $\lambda$ of $\bE$ are also $\mu$ linearly independent. Because if not then  there exists a vector $\bc$ such that $\frac{\partial \bc^\top \Gamma_M}{\partial \bz}\bF=\lambda \bc^\top \Gamma_M$. Hence, $\bc^\top \Gamma_M$ is precisely the nonlinear part of the eigenfunction corresponding to eigenvalue $\lambda$ and can be taken out of consideration to approximate. 
\end{assumption}
Typically, the measure $\mu$ can be taken to be equivalent to Lebesgue. 
Let $\cF_M={\rm span}\{ \gamma_1,\ldots,  \gamma_M\}$ and $\cF=\cL_2(\mu)$ for some positive measure $\mu$ on $\cZ$ with inner product $\left<f,g\right>_{\cL_2(\mu)}=\int_\cZ f(\bz)g(\bz) d\mu(\bz)$. 
\begin{definition} Given a countable set of functions $\Gamma=\{\gamma_j\}_{1}^\infty$, with $\gamma_j: \cZ\to \mR$, define ${\cal G}(\Gamma,\cZ)$ to be set of all linear combinations of elements in $\Gamma$ that converge {\color{black}$\mu$ almost everywhere in $\cZ$. }
\end{definition}
\noindent The main results of this section relies on the following Lemma and the proof follows along the lines of proof of  \cite[Lemma 14] {beard1997galerkin} on the Galerkin approximation. 
\begin{lemma}\label{lemma_galkerin}
If the set $\{\gamma_j\}_1^\infty$ is linearly independent and $\frac{\partial \gamma_j}{\partial \bz}\bF-\lambda \gamma_j\in {\cal G}(\Gamma,\cZ)$, then for all $M$,
\begin{align}
{\rm rank}(\bJ_\mu)=M\label{Jmudef},
\end{align}
where $\bJ_\mu:=\left< (\frac{\partial  \Gamma_M}{\partial \bz}\bF-\lambda \Gamma_M),  \Gamma_M\right>_m$\footnote{subscripts $m(v)$ are used to signify that matrix (vector) respectively.} 
\end{lemma}
\begin{proof} Since, $\{\gamma_j\}_1^\infty$ is linearly independent and $\frac{\partial \gamma_j}{\partial \bz}\bF-\lambda \gamma_j\in {\cal G}(\Gamma,\cZ)$, we have 
\[\frac{\partial \gamma_j}{\partial \bz}\bF-\lambda \gamma_j=\sum_{k=1}^\infty c_{kj}\gamma_j(\bz)=\bc_j^\top \Gamma(\bz).\]
We can write
{\small 
\begin{align}
&{\bJ_\mu}=\begin{pmatrix}\left<\bc_1^\top \Gamma,\gamma_1\right>&\hdots&\left<\bc_M^\top \Gamma,\gamma_1\right>\\
\vdots&\hdots&\vdots\\ \left<\bc_1^\top \Gamma,\gamma_M\right>&\hdots&\left<\bc_M^\top \Gamma,\gamma_M\right>
\end{pmatrix}\nonumber\\
=&\begin{pmatrix}\left<\Gamma,\Gamma_M\right>_m\bc_1&\hdots&\left<\Gamma, \Gamma_M\right>_m\bc_M\end{pmatrix}=\left<\Gamma, \Gamma_M\right>_m{\cal D}\label{Jmudefinition},
\end{align}
}
where ${\cal D}:=[\bc_1,\ldots,\bc_M]$, and hence ${\rm rank}({\bJ_\mu})={\rm rank}(\left<\Gamma, \Gamma_M\right>{\cal D})$. Now rank of $\left<\Gamma, \Gamma_M\right>$ is $M$ as $\{\gamma_j\}$ are linearly independent. Hence the proof will follow if we can show that rank of $\cal D$ is $M$. Following Assumption \ref{assume_basisfunction}, we know that the set $\{\frac{\partial \gamma_j}{\partial \bz }\bF-\lambda \gamma_j\}_1^\infty=\{\bc_j^\top\Gamma\}_1^\infty$ is linearly independent and hence the Gram matrix for $M$ of these vectors
\[\begin{pmatrix}\left<\bc_1^\top \Gamma,\bc_1^\top \Gamma\right>&\hdots&\left<\bc_1^\top \Gamma,\bc_M^\top \Gamma\right>\\\vdots&&\vdots\\\left<\bc_M^\top \Gamma,\bc_1^\top \Gamma\right>&\hdots&\left<\bc_M^\top \Gamma,\bc_M^\top \Gamma\right>
\end{pmatrix}={\cal D}^\top\left<\Gamma, \Gamma\right>_m{\cal D}\]
has rank equal to $M$. Now since $\Gamma$ is linearly independent, ${\rm rank}(\left<\Gamma,\Gamma\right>_m)=M$ and hence rank of $\cal D$ is equal to $M$. \qed
\end{proof}

\begin{theorem}\label{theorem_computation}  
Under Assumption \ref{assume_basisfunction}, the optimal Galerkin projection for the nonlinear term of the eigenfunction, i.e., $h(\bz)$ corresponding to eigenvalue $\lambda$ is given by $ \Gamma_M^\top(\bz) \Theta^\star$, where $ \Theta^\star$ is obtained as the unique solution of following linear equations.
\begin{align}\bJ_\mu\Theta=-\bb_\mu, \label{gproj}
\end{align}
where, $\bb_\mu:=\left<\bw^\top \bF_n,\Gamma_M\right>_v$ and $\bJ_\mu$ is as defined in (\ref{Jmudef}).
\end{theorem}
\begin{proof}
We seek to find a finite-dimensional approximation of $h(\bz)$ using the finite basis function $ \Gamma_M(\bz)$ as 
\[h(\bz)\approx \sum_k^M  u_k \bar \gamma_k(\bz)=\Gamma_M(\bz)^\top \Theta.\]
Substituting the above expression in (\ref{linear_nonlinear_eig}), we can write the approximation error as 
\begin{align}
{\rm error}=(\frac{\partial  \Gamma_M}{\partial \bz}\bF-\lambda \Gamma_M(\bz))^\top\Theta+\bw^\top \bF_n(\bz).
\end{align}
The objective is to determine the vector $\Theta$ such that the projection of the error on the finite basis $\{\gamma_k\}_{k=1}^M$  is equal to zero for all $\bz$. Hence, we obtain
\[\left<\left(\frac{\partial  \Gamma_M}{\partial \bz}\bF-\lambda  \Gamma_M(\bz)\right),\Gamma_M\right>_m\Theta=-\left<\bw^\top \bF_n(\bz), \Gamma_M\right>_v.\]
where the inner product is in $\cL_2(\mu)$. The uniqueness of solution follows from Lemma \ref{lemma_galkerin}, where the matrix $\bJ_\mu=\left<(\frac{\partial  \Gamma_M}{\partial \bz}\bF-\lambda  \Gamma_M(\bz)),\Gamma_M\right>_m$ is proved to have rank $M$ and hence invertible. The optimal solution $\Theta^\star$ is given by
\begin{align}\Theta^\star=-{\bJ_\mu}^{-1}\bb_\mu\label{solution1}\qed
\end{align}
\end{proof}

In general, computing the integral and inner product in $\cL_2(\mu)$ would be computationally expensive. Instead, the inner product can be replaced by a dot product using empirical measure. Let $\{\bz_1,\ldots, \bz_L\}$ be the finite data points drawn independent identically distributed (i.i.d) random w.r.t. $\mu$ from $\cZ$. Let $\hat \mu_L$ be the associated empirical measure i.e., $\hat \mu_L=\frac{1}{L}\sum_k^L \delta_{\bz_k}$, with $\delta_{\bz_k}$ the Dirac-delta measure. We then have $\int_\cZ f(\bz)d\hat \mu_L=\frac{1}{L}\sum_k^L f(\bz_k)$. With the use of this empirical measure, we can write the Eq. (\ref{gproj}) to determine the approximate coefficient vector $\Theta_L$ as the solution of finite linear equations.
\begin{align}
&\underbrace{\frac{1}{L}\sum_{k=1}^L \Gamma_M(\bz_k)\left(\frac{\partial  \Gamma_M}{\partial \bz}(\bz_k)\bF(\bz_k)-\lambda  \Gamma_M(\bz_k)\right)^\top}_{\bJ_{\hat \mu_L}}  \Theta_L\nonumber\\&=-\underbrace{\frac{1}{L}\sum_{k=1}^L \bw^\top \bF_n(\bz_k) \Gamma_M(\bz_k)}_{\bb_{\hat \mu_L}}\label{compute_eig}.
\end{align}
 Following Assumption \ref{assume_basisfunction}, the matrix $\bJ_{\hat\mu_L}$ will be invertible with probability one for $L\geq M$ as the samples $\{\bz_k\}_1^L$ are drawn i.i.d. The solution for $\Theta_L$ is given by 
 \begin{align}
 \Theta_L^\star=-\left(\bJ_{\hat \mu_L} \right)^{-1}\bb_{\hat \mu_L}\label{optforeigenfunction}.
 \end{align}
 We have following results.
 \begin{lemma}\label{lemma_eigenfunction} Let Assumption \ref{assume_basisfunction} hold true and
 $\psi_\lambda^{ML}(\bz)= \bw^\top \bz+ \Gamma_M^\top(\bz)\Theta_L^\star$ and $\psi_\lambda^{M}(\bz)= \bw^\top \bz+ \Gamma_M^\top(\bz)\Theta^\star$, where $\Theta^\star$ and $ \Theta_L^\star$ are the solutions of (\ref{solution1}) and (\ref{optforeigenfunction}) respectively. Then,
 \begin{align}
 \lim_{L\to \infty}\|\psi_\lambda^{ML}-\psi_\lambda^{M}\|=0\label{convergence_eigen},
 \end{align}
 where $\|\cdot\|$ is any norm on $\cF_M$.
 \end{lemma}
\begin{proof} Since the linear part, $\bw^\top \bz$ of the principal eigenfunction is common in both $\psi_{\lambda}^M$ and $\psi_\lambda^{ML}$, we only need to show (\ref{convergence_eigen}) for the nonlinear terms. We have by strong law of large number 
\begin{align}\lim_{L\to \infty} \Gamma_M^\top(\bz) \bJ_{\hat \mu_L}^{-1}\bb_{\hat \mu_L}=\Gamma_M^\top (\bz)\bJ_\mu^{-1}\bb_\mu.\label{lssolution}
\end{align}
with probability one since the matrix inverse operation is continuous and the samples are drawn i.i.d. \qed
\end{proof}
\begin{remark}\label{remark_convergencerate} Since the samples, $\{\bz_k\}$, are drawn i.i.d. random in the approximation of $\psi_\lambda^M$, the law of large number can be applied to study the convergence rate of $\psi_{\lambda}^{ML}\to \psi_\lambda^M$ as a function of data size $L$. Using the central limit theorem, this rate of convergence will be proportion to $\frac{1}{\sqrt{L}}$ \cite{schervish2014probability}.
\end{remark}

\section{Simulation Results}\label{section_simulation}

This section presents numerical examples based on procedures one and two for approximating the Lagrangian submanifolds. All the simulation codes are developed in  MATLAB  and run on a computer with 16 GB of RAM and a 3.8 GHz  Intel Core i7  processor. The total simulation time for each example was less than a minute. 
\subsection{Example 1} 
For our first example, procedures one and two can be carried out analytically as the principal (generalized) eigenfunctions of the uncontrolled system, and the Hamiltonian system can be computed analytically. The control dynamics of the two-dimensional example are given as follows. 
\begin{align}
\begin{pmatrix}\dot x_1\\\dot x_2\end{pmatrix}=\alpha\begin{pmatrix}-\cos x_2(x_1-2x_2)+4(x_1+\sin x_2)\\x_1-2x_2+2(x_1+\sin x_2)\end{pmatrix}+\begin{pmatrix}1\\0\end{pmatrix}u,\label{example1vf}
\end{align}
where $\alpha=\frac{1}{\cos  x_2+2}$.
We consider the quadratic state cost $q(\bx) = \frac{1}{2}((x_1-2x_2)^2+(x_1+\sin x_2)^2)$, and
 the control cost to be $\frac{1}{2}u^2$. For procedure one, we work with the principal eigenfunctions of the uncontrolled system which along with the eigenvalues are given by
 \begin{align}
    \Psi(\bx)= \begin{bmatrix}
        \phi_1\\
        \phi_2
    \end{bmatrix} = \begin{bmatrix}
        x_1-2 x_2\\
        x_1+\sin x_2
    \end{bmatrix},\;\;\;
    \Lambda=\begin{bmatrix}
        -1 & 0\\
         0  & 2 
    \end{bmatrix}.
\end{align}
In Fig. \ref{fig:eigsEx111}, we show the plots of the principal eigenfunctions of the uncontrolled system. 
Approximation \ref{approximation} made in procedure one for  this example is not an approximation as 
\[\bR(\bx)=\bB\bB^\top=\begin{pmatrix}1&0\\0&0\end{pmatrix},\frac{\partial \Phi}{\partial \bx}=\begin{pmatrix}1&-2\\1&\cos x_2\end{pmatrix},q(\bx)=\frac{1}{2}\Phi^\top\Phi,\]
and hence
\[\frac{\partial \Phi}{\partial \bx}\bB \bB^\top \frac{\partial \Phi}{\partial \bx}^\top=\begin{pmatrix}1&1\\1&1\end{pmatrix}=\bR_1, \;\;\bQ_1=\begin{pmatrix}1&0\\0&1\end{pmatrix}\]
are indeed  constant matrices. 
Based on the steps outlined for procedure one, the Lagrangian submanifold and the optimal cost function are given by
\begin{align}\bp^\star=\frac{\partial V}{\partial \bx}, \;\;V(\bx)=\frac{1}{2}\Phi^\top\bL \Phi\label{optimal_costex1}.
\end{align}
where, $\bL$ is obtained as the solution of Riccatti equation
(\ref{L_riccatti}) and the optimal control input equal to 
\begin{align}\bL= \begin{pmatrix}    0.49&  -0.62\\
   -0.62&    5.35\end{pmatrix}
,\bu^\star=- 4.61x_1 -0.263x_2 - 4.74\sin x_2\label{optimalcontrolex1}.
\end{align}
For procedure two, the eigenvalues of the Hamiltonian are given by $(2.14,1.19,-2.14,-1.19)$, and the principal eigenfunctions corresponding to the unstable eigenvalues are given by 
\[\begin{pmatrix}-0.20 \phi_1(\bx)-0.67 \phi_2(\bx)+0.67 P_1+0.20 P_2\\
-0.39 \phi_1(\bx)-0.11 \phi_2(\bx)+0.90 P_1+0.08 P_2\end{pmatrix},\]
where $(P_1,P_2)^\top=\bP=(\frac{\partial \Phi}{\partial \bx}^\top)^{-1}\bp$. We notice that the principal eigenfunctions of the Hamiltonian system are obtained as linear combinations of the principal eigenfunctions of the uncontrolled system. This will always be the case when the Approximation \ref{approximation} is no longer an approximation but is exact, as is the case with this example. 
The optimal cost and control for procedures 1 and 2 are the same for this example and are given in Eqs. (\ref{optimal_costex1}) and   (\ref{optimalcontrolex1}) respectively. In Fig. \ref{fig:v2Ex1}, we show the plot for the optimal cost and control input obtained using procedures 1, 2, and LQR control. The optimal cost and control input using LQR control is given by
\begin{align}
u_{lqr}=- 5.06 x_1 - 5.74 x_2,\;V_{lqr}=\frac{1}{2}\bx^\top\begin{pmatrix}2.53&    2.87\\
    2.87&    7.59\end{pmatrix}\bx.
\end{align}
\vspace{-0.2in}
\begin{figure}[htbp]
\centering
\includegraphics[width=3.5in]{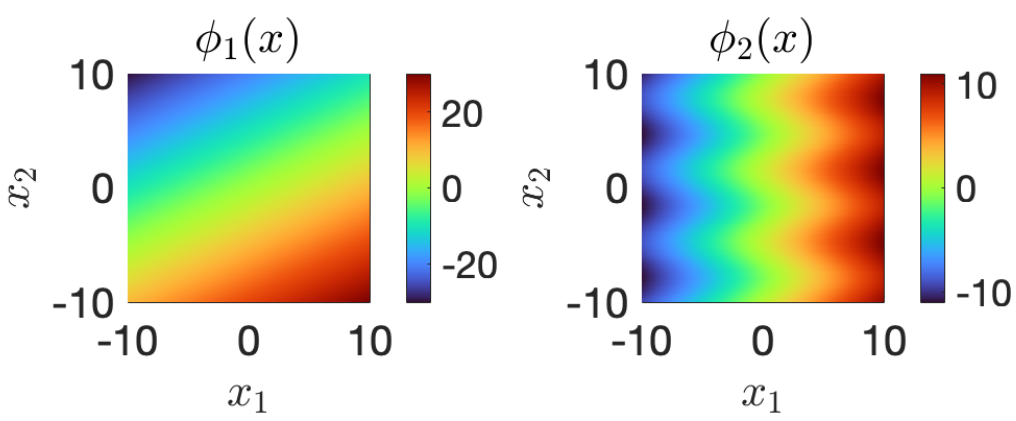}
\vspace{-0.2in}
\caption{Principal Eigenfunctions $\phi_1(\bx),\phi_2(\bx)$ of the uncontrolled system}
\label{fig:eigsEx111}
\end{figure}

\begin{figure}[htbp]
\centering
\includegraphics[width=3.4in]{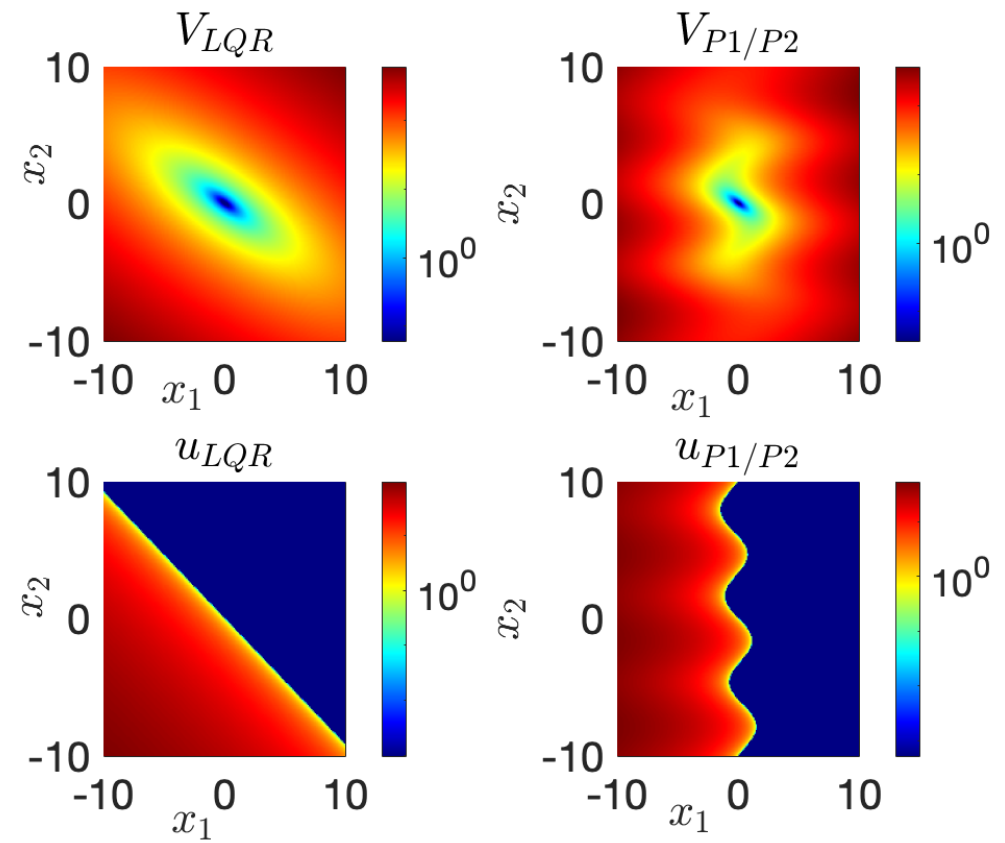}
\vspace{-0.05in}
\caption{Value function of LQR control and Proc. 1 and Proc.2 (top). Feedback contorol law for LQR, Proc. 1, and Proc.2 (bottom).}
\label{fig:v2Ex1}
\end{figure}
We compare the analytical results presented above with those obtained using our proposed framework. In particular, we use the approach outlined in Section \ref{section_compute} to approximate the Koopman principal eigenfunctions and the HJ solution. We also compare the results obtained using our proposed framework with one obtained using the Taylor series approximation \cite{navasca2000solution,al1961optimal}. We use the polynomial basis function to approximate the Koopman principal eigenfunctions for a fair comparison with the Taylor series-based solution. For the Taylor series approximation of the HJ solution, we Taylor expand the vector field (\ref{example1vf}) around the origin as 
\[\dot \bx=\bff(\bx)+\bB u=\bA\bx+\bF_2(\bx)+\bF_3(\bx)+\cdots+\bB u.\]
where $\bF_j(\bx)$ is the part of the vector field containing nonlinear terms of order $j$. Similarly, the value function, $V(\bx)$, the control input, $k(\bx)$, and the cost function, $q(\bx)$, are expanded in Taylor series as follows:
\begin{align*}V(\bx)=\sum_{j=2}^\infty V_j(\bx),\;u=\bK\bx+\sum_{j=2}^\infty k_j(\bx),\;q(\bx)=\sum_{j=2}^\infty q_j(\bx)
\end{align*}
where $V_2(\bx)=\bx^\top \bP \bx$ and  $\bK=-\bB^\top \bP$ with $\bP$ being the solution of Riccatti equation. $V_j(\bx)$ and $k_j(\bx)$ are the cost function and control input containing nonlinear terms of order $j$. In \cite[Eq. 2.21]{navasca2000solution} \cite[Eq. 2.9]{al1961optimal}, it has been  shown that the $V_{j>2}(\bx)$ satisfy a recursive equation, where $V_j$ is function of $V_k(\bx)$ with $k<j$, $\bF_\ell(\bx)$ for $\ell\leq j$ and $q_j(\bx)$. For more details on this derivation, refer to \cite{navasca2000solution,al1961optimal}. In Fig. \ref{fig_eigenerror}, we show the normalized error plot between the analytical Koopman principal eigenfunction, $\phi_2(\bx)$, of the uncontrolled system and the one, approximated from data with finite basis, $\phi_{2}^{ML}(\bx)$, i.e., $\frac{\|\phi_{2}^{ML}-\phi_2\|}{\|\phi_2\|}$ as the function of data length, $L$. The results are presented using the boxplots where the red line denotes the mean and the $75^{th}$ and $25^{th}$ quartiles are represented by the upper and lower edges of the black box, and the red stars represent the outliers.
\vspace{-0.0in}
\begin{figure}
\vspace{-0.2in}
  \begin{center}
    \includegraphics[width=0.3\textwidth]{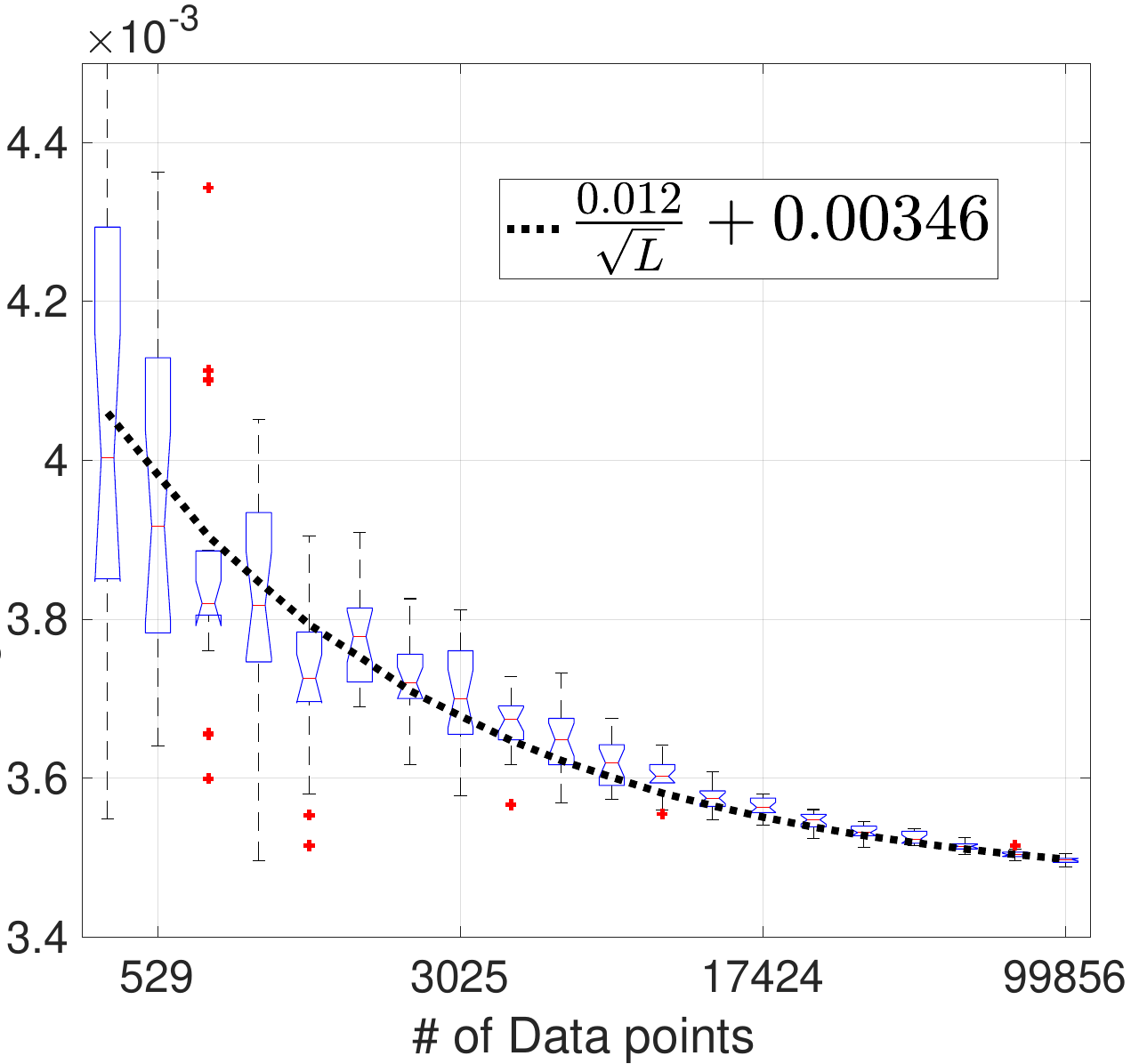}
  \end{center}
  \vspace{-0.1in}
  \caption{Principal Eigenfunction error vs data length}\label{fig_eigenerror}
\end{figure}
We choose the polynomial basis function of degree five and hence $M=18$.
The plot shows the error decays at the rate of $\frac{1}{\sqrt{L}}$ (Remark \ref{remark_convergencerate}) and validates the results of Lemma \ref{lemma_eigenfunction}. Now, let $V^\star(\bx)$, $V_{T}(\bx)$, and $V_{K}$ be the analytical, Taylor series-based, and Koopman spectrum-based approximated HJ solutions, respectively. In Fig. \ref{fig_valuecomparsion} we show the plot for the point-wise error between $|V^\star(\bx)-V_{T}(\bx)|$ and $|V^\star (\bx)-V_{K}(\bx)|$. The error with the Taylor series-based approximated solution is an order of magnitude higher than the proposed Koopman spectrum-based HJ solution. This demonstrates the effectiveness of the developed framework over the solution obtained using the Taylor series. 

\begin{figure}[htbp]
\includegraphics[width=3.7in]{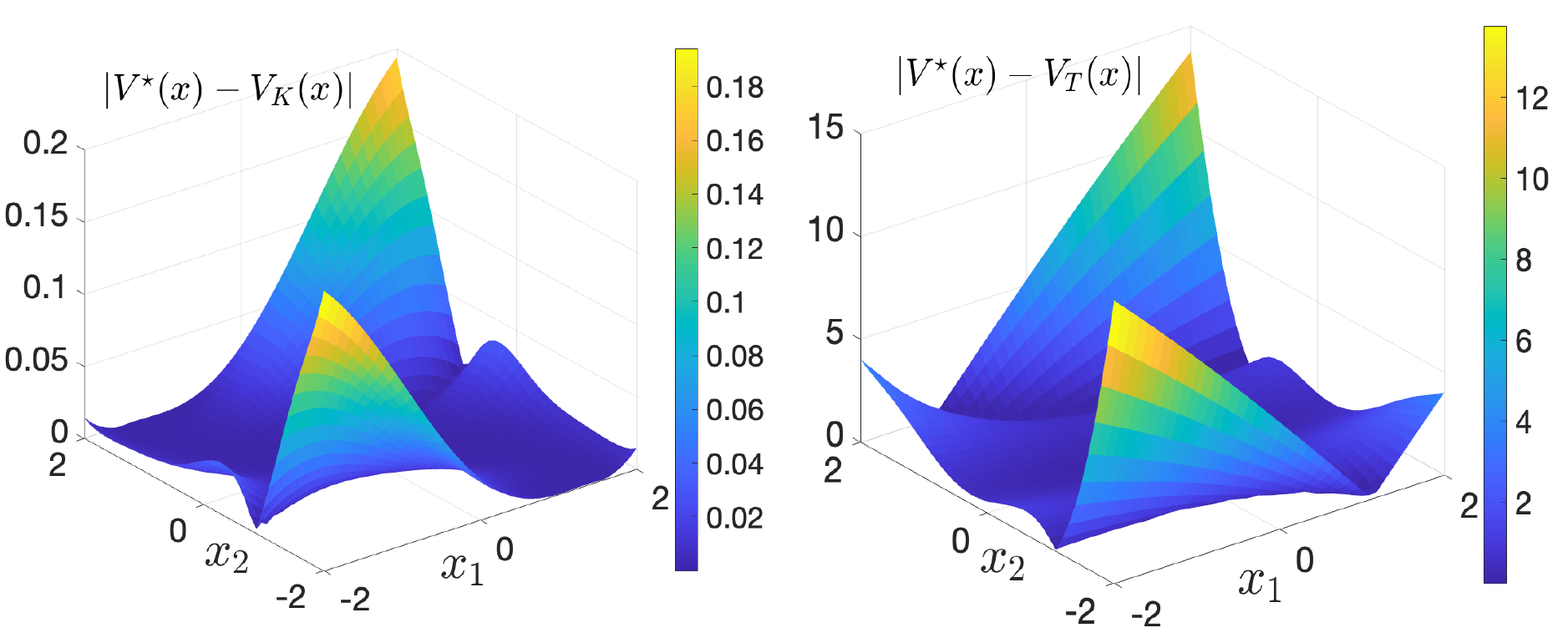}
\vspace{-0.1in}
\caption{Comparison of optimal cost function error: $|V^\star(\bx)-V_K(\bx)|$ (left) and $|V^\star(\bx)-V_T(\bx)|$ (right).}
\label{fig_valuecomparsion}
\end{figure}
\vspace{-0.2in}

\subsection{Example 2}
Our next example is an inverted pendulum on a cart. After removing the cart position state, the reduced order dynamics while retaining angular position, $\theta$, angular velocity, $\psi$, of the pendulum and cart velocity $\vartheta$ as states are of the form.
\begin{align}
\dot \theta&=\psi\nonumber\\
\begin{pmatrix} \dot \psi\\\dot \vartheta \end{pmatrix}&=\bM^{-1}\begin{pmatrix}-b\vartheta +m\ell \psi^2 \sin(\theta-\pi)+u\\-mg\ell \sin(\theta-\pi)\end{pmatrix}
\end{align}
where
\[\bM:=\begin{pmatrix}
m\ell\cos (\theta-\pi)&(M+m)\\(I+m\ell^2)&m\ell\cos(\theta-\pi).
\end{pmatrix}\]
Let $\bx=(\theta,\psi,\vartheta)^\top$. The parameter values are chosen to be 
$M=0.5 {\rm kg}, m=0.2 {\rm kg}, b=0.1 {\rm N/m/sec}, l= 0.3 {\rm m}, I =     0.006 {\rm kg.m^2}$. The objective is to stabilize the unstable equilibrium point at the origin optimally. We use procedure one to design the optimal control for this example, which requires computing the eigenfunctions of the uncontrolled system. The cost function is chosen to be a quadratic function of states $\bx^\top \bx$ and control input $u^2$. We used the methodology developed in Section \ref{section_compute} to approximate the principal eigenfunctions. We used $1e^4$ initial condition uniformly distributed over the domain $[-3,3]\times [-5,5]\times[-5,5]$ to approximate the principal eigenfunctions. The polynomial basis functions of maximum degree two are used to approximate the Koopman principal eigenfunctions of the uncontrolled system. Figures \ref{figiv1}-\ref{figiv4},  show the plots for the trajectories of the closed-loop system starting from  10 different initial conditions distributed around $\bx=(0.7, -4.2,6.2)^\top$ and the corresponding control inputs. We compare the closed-loop system trajectories obtained using our proposed approach against the linear quadratic regulator and pole placement-based controllers. These comparison plots show that the controller obtained using our proposed approach can successfully stabilize the unstable equilibrium point at the origin. In contrast, the controller obtained using LQR and pole-placement approach failed to stabilize the equilibrium point. 
\begin{figure}[htbp]
\centering
\includegraphics[width=2.8in]{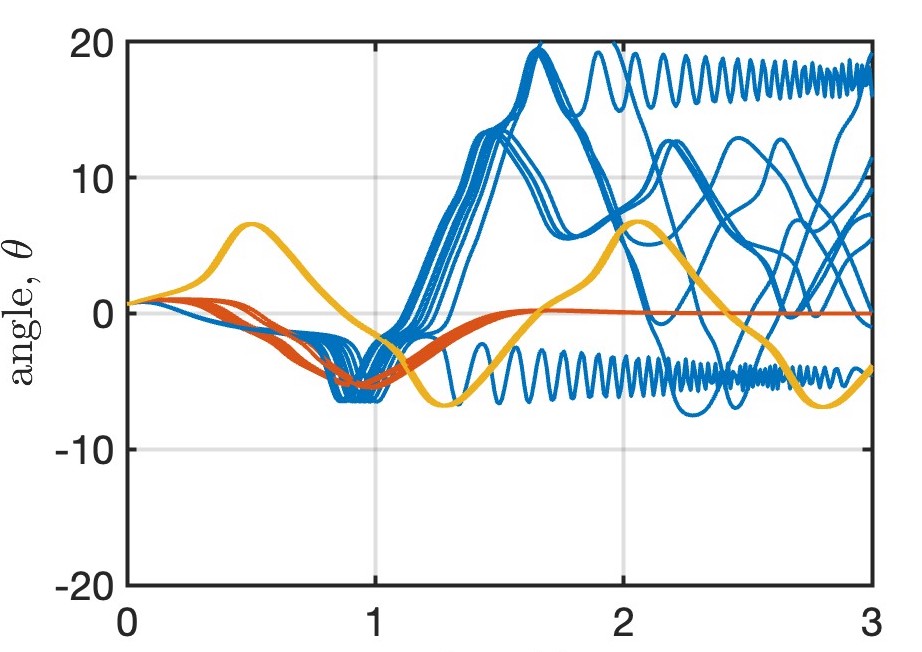}
\vspace{-0.05in}
\caption{Comparison of angular position  obtained using Procedure one (red), LQR controller (blue), and pole placement controller (yellow) from multiple initial conditions.}
\label{figiv1}
\end{figure}
\vspace{-0.1in}
\begin{figure}[htbp]
\centering
\includegraphics[width=2.8in]{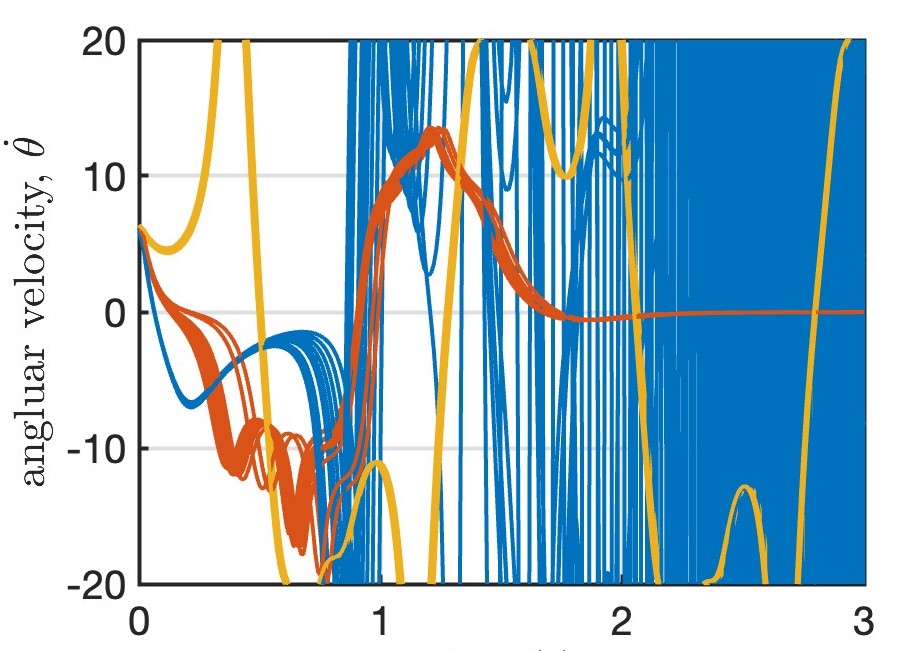}
\vspace{-0.05in}
\caption{Comparison of angular velocity  obtained using Procedure one (red), LQR controller (blue), and pole placement controller (yellow) from multiple initial conditions.}
\label{figiv2}
\end{figure}
\vspace{-0.1in}
\begin{figure}[htbp]
\centering
\includegraphics[width=2.8in]{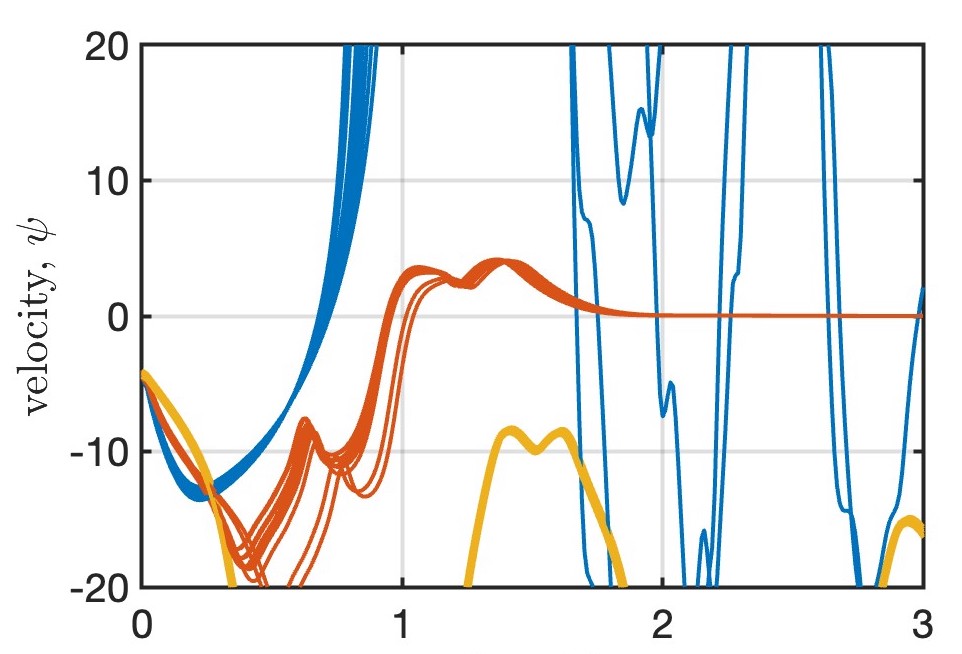}
\vspace{-0.05in}
\caption{Comparison of cart velocity  obtained using Procedure one (red), LQR controller (blue), and pole placement controller (yellow) from multiple initial conditions.}
\label{figiv3}
\end{figure}
\vspace{-0.1in}
\begin{figure}[htbp]
\centering
\includegraphics[width=2.8in]{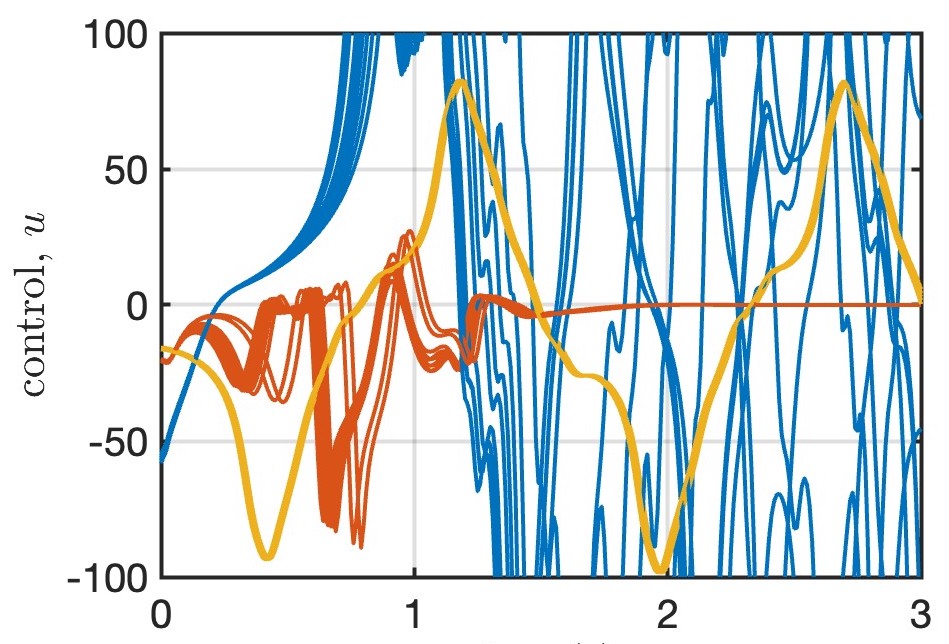}
\vspace{-0.05in}
\caption{Comparison of  control inputs obtained using Procedure one (red), LQR controller (blue) and, pole placement controller (yellow) from multiple initial conditions.}
\label{figiv4}
\end{figure}
\section{Conclusions}\label{section_conclusions}
This paper combines tools from linear Koopman operator theory and differential geometry to provide a novel approach for the approximate computation of the HJ solution. We present two different procedures based on the spectrum of the Koopman operator for the computation of the Lagrangian submanifold, which are instrumental in solving the HJ equation.  
Furthermore, given the significance of the HJ equation in various problems in systems theory, the proposed methodology involving the Koopman operator can be extended in several different directions, including analysis and synthesis of a control system and robust control. The proposed approach is also the most promising direction for extending existing results in linear control theory to the nonlinear control system by exploiting the spectral properties of the Koopman operator. Future research efforts will focus on implementing the two procedures in the data-driven setting. 
\vspace{-0.1in}
\section*{Appendix}



\noindent\begin{proof}
({\it Proof of Proposition \ref{Proposition_Lagsubspace}}).
Following Assumption  \ref{assumption_hamiltonianeigenvalues}, we know that the Hamiltonian matrix, $\cal H$, has no eigenvalues on the $j\omega$ axis. Let $\bd_1^\top,\ldots,\bd_n^\top$ be the left eigenvectors associated with eigenvalues with positive real part. Following Corollary \ref{proposition_mainfolds}, the  stable subspace of the linear system is characterized by the joint zero-level sets of Koopman principal eigenfunctions corresponding to eigenvalues with a positive real part, and hence the stable subspace is characterized as
\[\mE_s=\{(\bar  \bX,\bar\bP): \bD\begin{pmatrix}\bar \bX\\\bar \bP\end{pmatrix}=\bD_1\bar \bX+\bD_2 \bar \bP=0\},\]
where the matrix $\bD,\bD_1,\bD_2$ are as defined in (\ref{D_def}). We seek the parameterization of the stable subspace in terms of $\bar \bX$, and hence we write $\bar \bP=\bL\bar \bX$ for some symmetric matrix $\bL$. Symmetric $\bL$ follows from the fact that we seek parameterization of the form $\bar\bP=\frac{\partial V(\bar \bX)}{\partial \bar \bX}$, i.e., a gradient of a scalar value function $V=\frac{1}{2}\bar \bX^\top \bL\bar \bX$. Substituting for $\bar \bP$ in the stable subspace equation, we obtain
\begin{align}\bD_1\bar \bX+\bD_2 \bL\bar \bX=(\bD_1 +\bD_2 \bL)\bar\bX=0.
\end{align}
Since the above has to be true for all $\bar \bX$, we need $\bD_1+\bD_2 \bL=0$. Following Assumption \ref{assumption_hamiltonianeigenvalues},  (\ref{assume_complimentary}) is equivalent to
\begin{align}
\mE_s^\perp\oplus \begin{pmatrix}I\\0\end{pmatrix}=\mR^{2n}.\label{condition}
\end{align}
Now since matrix $\bD$ span the space orthogonal to $\mE_s$ i.e., $\mE_s^\perp$,  (\ref{condition}) is equivalent to the invertability of $\bD_2$ matrix
Hence, we obtain $\bL=\bD_2^{-1}\bD_1$. This gives us the required expression for the Lagrangian subspace in $\bar\bX,\bar\bP$ coordinates. 
\qed
\end{proof}

\begin{proof}({\it Proof of Proposition \ref{prop_riccati_lagsubspace}})
Since, the matrix $\bD$ from Proposition \ref{Proposition_Lagsubspace} consists of left eigenvectors with eigenvalues with positive real part, say $\bar \Lambda$, of the Hamiltonian $\cal H$, it forms an invariant subspace and satisfies
\begin{align}
\begin{pmatrix}\bD_1&\bD_2\end{pmatrix}\begin{pmatrix}\Lambda&-\hat \bR\\-\hat \bQ&-\Lambda^\top\end{pmatrix}=\bar \Lambda\begin{pmatrix}\bD_1&\bD_2\end{pmatrix}.    \end{align}

Premultiplying by $\bD_2$ inverse, we obtain
\begin{align}
\begin{pmatrix}-\bL&I\end{pmatrix}\begin{pmatrix}\Lambda&-\hat \bR\\-\hat \bQ&-\Lambda^\top\end{pmatrix}=\bD_2{-1}\bar \Lambda\bD_2\begin{pmatrix}-\bL&I\end{pmatrix}.  
\end{align}
Postmultiply by $\begin{pmatrix}I\\\bL\end{pmatrix}$, we obtain
\begin{align}
    \bL \Lambda +\Lambda^\top \bL-\bL\hat \bR\bL+\hat \bQ=0\label{riccattiL}.
\end{align}
and 
\begin{align*}
    \Lambda - \hat \bR \bL 
    &= -\bD_2^\top \bar \Lambda^\top \left(\bD_2^{-1}\right)^\top 
    \implies \sigma(\Lambda - \hat \bR \bL) &= \sigma(-\bar \Lambda). \qed
\end{align*}
\end{proof}
\bibliographystyle{IEEEtran}
\bibliography{ref1,ref}
\vspace{-0.1in}
\begin{IEEEbiography}[{\includegraphics[width=1in,height=1.25in,clip,keepaspectratio]{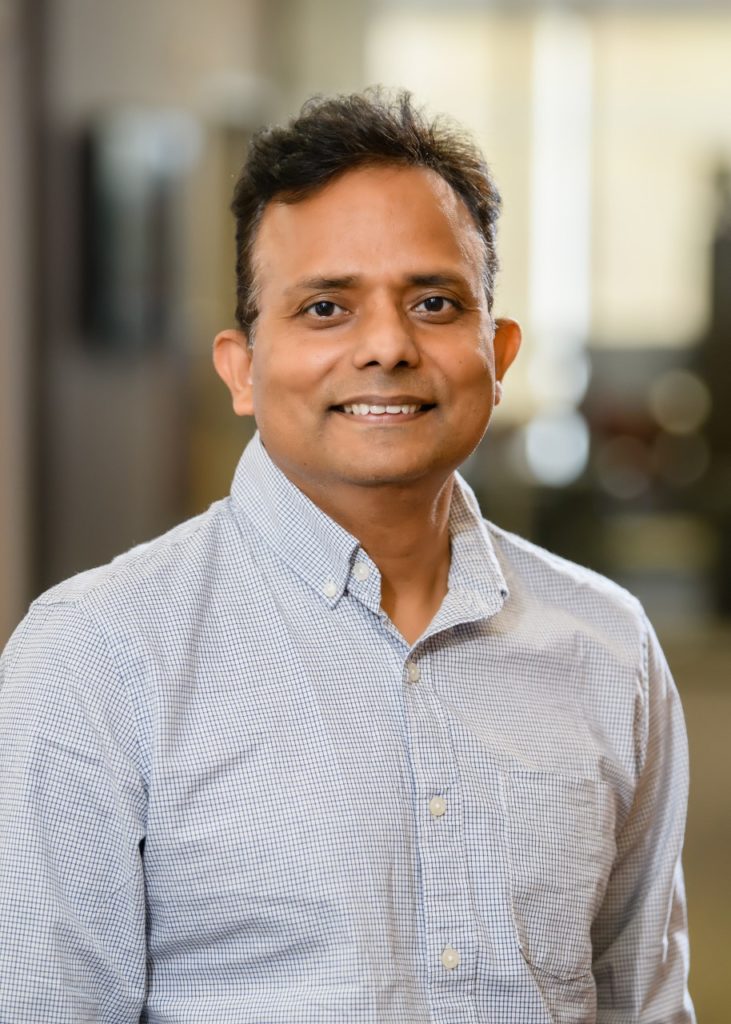}}]%
{Umesh Vaidya}(M’07, SM'19)  received the Ph.D. degree in mechanical engineering from the University of California at Santa Barbara, Santa Barbara, CA, in
2004. He was a Research Engineer at the United Technologies Research Center (UTRC), East Hartford, CT, USA. He is currently a professor in the Department of Mechanical Engineering, Clemson University, S.C., USA. Before joining Clemson University in 2019, and since 2006, he was a faculty with the department of Electrical and Computer Engineering at Iowa State University. He is the recipient of 2012 National Science Foundation CAREER award. His current research interests include dynamical systems and control theory with applications to power grid and robotics. 
\end{IEEEbiography}

\end{document}